\documentclass[a4paper]{amsart}
\usepackage{amscd, amsfonts, amssymb}
\usepackage[mathscr]{eucal}
\usepackage{fullpage}
\usepackage{tikz}
\usepackage[all,cmtip]{xy}
\usepackage{verbatim}

\theoremstyle{plain}
\numberwithin{equation}{section}
\newtheorem{thm}[equation]{Theorem}
\newtheorem{lem}[equation]{Lemma}
\newtheorem{cor}[equation]{Corollary}
\newtheorem{prop}[equation]{Proposition}
\newtheorem{defn}[equation]{Definition}

\theoremstyle{remark}
\newtheorem{rem}[equation]{Remark}
\newtheorem{rems}[equation]{Remarks}
\newtheorem{ex}[equation]{Example}

\newcommand{\NN}{\mathbb{N}}

\newcommand{\inverse}{^{-1}}
\newcommand\ra{\rightarrow}
\newcommand\lra{\longrightarrow}

\newcommand{\ind}{{\rm ind}}
\newcommand{\iso}{\cong}
\newcommand{\Gg}{{\mathfrak g}}
\newcommand{\hh}{{\mathfrak h}}

\DeclareMathOperator\Ad{Ad}
\DeclareMathOperator\GL{GL}
\DeclareMathOperator\Mat{Mat}
\DeclareMathOperator\SL{SL}
\DeclareMathOperator\SP{Sp}

\DeclareMathOperator\Aut{Aut}

\DeclareMathOperator\Char{char}
\DeclareMathOperator\Lie{Lie}
\DeclareMathOperator\soc{soc}
\DeclareMathOperator\End{End}

\newcommand{\ovl}{\overline}


\title{Orbit Closures and Invariants}

\author[M.\ Bate]{Michael Bate}
\address
{Department of Mathematics,
University of York,
York YO10 5DD,
United Kingdom}
\email{michael.bate@york.ac.uk}

\author[H.\ Geranios]{Haralampos Geranios}
\address
{Department of Mathematics,
University of York,
York YO10 5DD,
United Kingdom}
\email{haralampos.geranios@york.ac.uk}

\author[B.\ Martin]{Benjamin Martin}
\address
{Department of Mathematics,
University of Aberdeen,
King's College,
Fraser Noble Building,
Aberdeen AB24 3UE,
United Kingdom}
\email{b.martin@abdn.ac.uk}

\date{\today}

\begin{document}

\begin{abstract}
Let $G$ be a reductive linear algebraic group, $H$ a reductive subgroup of $G$ and $X$ an affine $G$-variety.
Let $X^H$ denote the set of fixed points of $H$ in $X$, and $N_G(H)$ the normalizer of $H$ in $G$.
In this paper we study the natural map of quotient varieties $\psi_{X,H}\colon X^H/N_G(H) \to X/G$ induced by the inclusion $X^H \subseteq X$.
We show that, given $G$ and $H$, $\psi_{X,H}$ is a finite morphism for all affine $G$-varieties $X$ if and only if $H$ is a $G$-completely reducible subgroup of $G$
(in the sense defined by J-P. Serre); this was proved in characteristic $0$ by Luna in the 1970s.  We discuss some applications and give a criterion for $\psi_{X,H}$ to be an isomorphism.  We show how to extend some other results in Luna's paper to positive characteristic and also prove the following theorem.  Let $H$ and $K$ be reductive subgroups of $G$; then the double coset $HgK$ is closed for generic $g\in G$ if and only if $H\cap gKg^{-1}$ is reductive for generic $g\in G$.
\end{abstract}

\maketitle

\section{Introduction}\label{sec:intro}

The purpose of this paper is to establish some results in geometric invariant theory over fields of positive characteristic, where tools from characteristic 0---such as Luna's \'Etale Slice Theorem---are not available.  In particular, we prove the following theorem and give some applications (see Section~\ref{sec:notation} for precise definitions of terms).  Let $k$ be an algebraically closed field of characteristic $p\geq 0$.

\begin{thm}
\label{thm:mainthm}
Suppose $G$ is a reductive linear algebraic group over $k$ and $H$ is a reductive subgroup of $G$.
Then the following are equivalent:
\begin{itemize}
\item[(i)] $H$ is $G$-completely reducible;
\item[(ii)] $N_G(H)$ is reductive and, for every affine $G$-variety $X$, the natural map of quotients $\psi_{X,H}:X^H/N_G(H) \to X/G$ is a finite morphism (here $X^H$ denotes the $H$-fixed points in $X$).
\end{itemize}
\end{thm}

\noindent The study of closed orbits is central in geometric invariant theory---the closed orbits for $G$ in $X$ parametrise the points of the quotient variety $X/G$.  An important piece of the proof of Theorem~\ref{thm:mainthm} is Proposition~\ref{prop:htog}, which gives a connection between the closed $G$-orbits in $X$ and the closed $H$-orbits in $X$; cf.\ \cite{luna}, \cite{rich1}, \cite{bateopt} and \cite{BMRT}, for example.

Theorem \ref{thm:mainthm} reduces to the main result in Luna's paper \cite{luna} when $k$ has characteristic $0$, because condition (i) and the first hypothesis of (ii) are automatic in characteristic $0$ if $H$ is already assumed to be reductive.  Luna's methods use the powerful machinery of \'etale slices, based on his celebrated ``\'Etale Slice Theorem'' \cite{luna0}; see Section~\ref{subsec:slices} below for more on \'etale slices.  Many useful consequences flow from the existence of an \'etale slice (see Proposition~\ref{prop:locstab} below, for example).  Although \'etale slices sometimes exist in positive characteristic \cite{BaRi}, in general they do not.  Our methods differ from Luna's in that they apply equally well in all characteristics.
These methods also allow us to provide extensions to positive characteristic of other results from \cite{luna} (see Proposition~\ref{prop:finquotfin}, Remark~\ref{rems:failure}(i) and Proposition~\ref{prop:luna2}).

The work in this paper fits into a broad continuing programme of taking results about algebraic groups and related structures from
characteristic zero and proving analogues in positive characteristic.  A basic problem with this process is that results---such as the existence of an \'etale slice---that are true when $p=0$ may simply fail when $p>0$ (cf.\ Examples~\ref{ex:noslice}, \ref{ex:sl2}, \ref{ex:harry1} and \ref{ex:harry2}); a further illustration of this in the context of this paper is that a reductive group may fail to be linearly reductive
(recall that a linear algebraic group is called \emph{reductive} if it has trivial unipotent radical, and \emph{linearly reductive}
if all its rational representations are semisimple).
When $p=0$, a connected group is linearly reductive if and only if it is reductive, whereas if $p>0$
a connected group is linearly reductive if and only if it is a torus \cite{nagata}.
Even if a result remains true in positive characteristic, it may be much harder to prove,
an example here being the problem of showing that the ring of invariants $R^G$ is finitely generated,
where $R$ is a finitely-generated $k$-algebra and $G \subseteq \Aut(R)$ is reductive.
This was resolved in characteristic $0$ in the 1950s, but not in positive characteristic until the 1970s
(see the introduction to Haboush's paper \cite{haboush0}).

In some contexts in positive characteristic where the hypothesis of reductivity is too weak and linear reductivity is too strong,
it has been found that a third notion, that of $G$-complete reducibility, provides a good balance (cf.\ \cite[Cor.\ 1.5]{martin3}) and our main theorem is another example of this phenomenon. See Section \ref{subsec:cr} for the definition. The idea is that when $p=0$ there is no distinction between demanding that
a subgroup $H$ of a reductive group $G$ is reductive or linearly reductive or $G$-completely reducible,
but there is a huge difference in positive characteristic.
The notion of complete reducibility was introduced by Serre \cite{serre} and Richardson\footnote{Richardson
originally defined \emph{strong reductivity} for subgroups of $G$, but his notion was shown to be equivalent to Serre's in \cite[Thm.\ 3.1]{BMR}.}
\cite{rich3}, and over the past twenty years or so has found many applications in the theory of
algebraic groups, their subgroup structure and representation theory, geometric invariant theory, and the theory of buildings:
for examples, see \cite{bateopt}, \cite{BMR}, \cite{GIT}, \cite{ls}, \cite{martin1}, \cite{martin2}, \cite{mcninch}, \cite{stewartG2}, \cite{stewartF4}, \cite{stewart}.

The paper is set out as follows.  Section~\ref{sec:notation} contains preparatory material from geometric invariant theory and the theory of complete reducibility.  The proof of Theorem~\ref{thm:mainthm} contains three main ingredients, each dealt with in a separate section.  In Section~\ref{sec:prelims} we build on work of Bardsley and Richardson to establish the important technical result Proposition~\ref{prop:finquotfin}, which gives a criterion for a map of quotient varieties to be finite.  In Section~\ref{sec:lunapartone} we carry out our analysis of the closed $G$- and $H$-orbits and show that $\psi_{X,H}$ is quasi-finite if $H$ is $G$-completely reducible (Theorem~\ref{thm:quasifinite}).  In Section~\ref{sec:lunaparttwo} we show that the image of $\psi_{X,H}$ is closed (Theorem~\ref{thm:closed}).  The key idea here is to consider the map of projectivisations ${\mathbb P}(X^H)\ra {\mathbb P}(X)$ induced by the inclusion of $X^H$ in $X$ when $X$ is a $G$-module; the $G$-complete reducibility of $H$ guarantees that we get a well-defined map of quotient varieties ${\mathbb P}(X^H)/G\ra {\mathbb P}(X)/G$.  Section~\ref{sec:lunapartthree} draws these strands together and completes the proof of Theorem~\ref{thm:mainthm} using Proposition~\ref{prop:insep_ZMT} (a variation on Zariski's Main Theorem).

Section~\ref{sec:sep} gives a criterion for $\psi_{X,H}$ to be an isomorphism onto its image (Theorem~\ref{thm:normaliso}).  In Section~\ref{sec:ex} we use representation theory to construct some examples relevant to Theorem~\ref{thm:mainthm}.  In Section~\ref{sec:dblecosets} we give a criterion (Theorem~\ref{thm:closedcrit}) for generic double cosets $HgK$ of $G$ to be closed, where $H$ and $K$ are reductive subgroups of $G$.  Luna proved a stronger result \cite{luna00} in characteristic 0 using \'etale slice methods, but our techniques work when \'etale slices are not available.  We give some applications of Theorem~\ref{thm:closedcrit} (Examples~\ref{ex:B2} and \ref{ex:E7}); these serve as applications of Theorem~\ref{thm:mainthm} as well.  We finish in Section~\ref{sec:Gcr} by using the theory we have developed to prove some results on complete reducibility.

\medskip
\noindent
\emph{Acknowledgements}:
The first author would like to thank Sebastian Herpel for the conversations we had which led to the first iteration of some of the ideas in this paper,
and also Stephen Donkin for some very helpful nudges towards the right literature.
All three authors acknowledge the funding of EPSRC grant  EP/L005328/1.
We would like to thank the anonymous referee for their very insightful comments and for pointing out a subtle gap in the proof of Theorem~\ref{thm:mainthm}.


\section{Notation and Preliminaries}\label{sec:notation}

\subsection{Notation}
Our basic references for the theory of linear algebraic groups are the books \cite{borel} and \cite{springer}.
Unless otherwise stated, we work over a fixed algebraically closed field $k$ with no restriction on the characteristic.  By a variety we mean a quasi-projective variety over $k$, and we identify a variety $X$ with its set of $k$-points.
For a linear algebraic group $G$ over $k$, we let $G^0$ denote the connected component of $G$ containing the identity element $1$
and $R_u(G) \unlhd G^0$ denote the unipotent radical of $G$.
We say that $G$ is \emph{reductive} if $R_u(G) = \{1\}$; note that we do not require a reductive group to be connected.
When we discuss subgroups of $G$, we really mean \emph{closed} subgroups;
for two such subgroups $H$ and $K$ of $G$, we set $HK:=\{hk\mid h \in H, k\in K\}$.
We denote the centralizer of a subgroup $H$ of $G$ by $C_G(H)$, and the normalizer by $N_G(H)$.  All group actions are left actions unless otherwise indicated.

We make repeated use of the following result \cite[Lem.\ 6.8]{martin1}: if $G$ is reductive and if $H$ is a reductive subgroup of $G$ then $N_G(H)^0= H^0C_G(H)^0$.

Given a linear algebraic group $G$, let $Y(G)$ denote the set of cocharacters of $G$, where a \emph{cocharacter} is a homomorphism
of algebraic groups $\lambda:k^* \to G$.
Note that since the image of a cocharacter is connected, we have $Y(G) = Y(G^0)$.
A linear algebraic group $G$ acts on its set of cocharacters: for $g \in G$, $\lambda \in Y(G)$ and $a \in k^*$, we
set $(g\cdot\lambda)(a) = g\lambda(a)g\inverse$.

Given an affine variety $X$ over $k$, we denote the coordinate ring of $X$ by $k[X]$ and the function field of $X$ (when $X$ is irreducible) by $k(X)$.
Given $x \in X$, we let $T_x(X)$ denote the tangent space to $X$ at $x$.
Recall that for a linear algebraic group $G$, $T_1(G)$ has the structure of a Lie algebra, which we also denote by $\Lie(G)$ or $\Gg$.
Given a morphism $\phi:X \to Y$ of affine varieties $X$ and $Y$ and a point $x \in X$, we let
$d_x\phi:T_x(X) \to T_{\phi(x)}(Y)$ denote the differential of $\phi$ at $x$.
We say that $X$ is a \emph{$G$-variety} if the linear algebraic group $G$ acts morphically on $X$.
If $X$ is affine then the action of $G$ on $X$ gives a linear action of $G$ on $k[X]$, defined by $(g\cdot f)(x) = f(g\inverse\cdot x)$
for all $g \in G$, $f \in k[X]$ and $x \in X$.
Given a $G$-variety $X$ and $x \in X$, we denote the $G$-orbit through $x$ by $G\cdot x$ and
the stabilizer of $x$ in $G$ by $G_x$.
If $x,y \in X$ are two points on the same $G$-orbit, then we sometimes say $x$ and $y$ are \emph{$G$-conjugate}.  For $x\in X$, we denote the orbit map $G\ra G\cdot x$, $g\mapsto g\cdot x$ by $\kappa_x$; we say the orbit $G\cdot x$ is {\em separable} if $\kappa_x$ is separable.
We denote by $X^G$ the set of $G$-fixed points in $X$, and by $k[X]^G$ the ring of $G$-invariant functions in $k[X]$.

Given a morphism of varieties $f\colon V\ra W$, define $e(v)$ for $v\in V$ to be ${\rm max}({\rm dim}(Z))$, where $Z$ ranges over the irreducible components of $f^{-1}(f(v))$ that contain $v$.  By \cite[AG.10.3]{borel}, $e(v)$ is an upper semi-continuous function of $v$.  This implies the following useful result about dimensions of stabilizers for a $G$-variety $X$ \cite[Lem.~3.7(c)]{newstead}: for any $r\in {\mathbb N}\cup \{0\}$, the set $\{x\in X\mid {\rm dim}(G_x)\geq r\}$ is closed.  We deduce the lower semi-continuity of orbit dimension: that is, for any $r\in {\mathbb N}\cup \{0\}$, the set $\{x\in X\mid {\rm dim}(G\cdot x)\leq r\}$ is closed.  In particular, the set $\{x\in X\mid \mbox{${\rm dim}(G\cdot x)$ is maximal}\}$ is open.  We also need an infinitesimal version of these results.  Given a variety $Z$, we denote the (reduced) tangent bundle of $Z$ by $TZ$; we may identify $TZ$ with the set of pairs $\{(z,v)\mid z\in Z, v\in T_z(Z)\}$, and we have a canonical embedding from $Z$ to $TZ$ given by $z\mapsto (z,0)$.  (The tangent bundle is constructed in \cite[AG.16]{borel} as a possibly non-reduced scheme over $k$; here we take the tangent bundle to be the corresponding reduced scheme.)
If $\psi\colon Z\ra W$ is a morphism of varieties then we have a map $d\psi\colon TZ\ra TW$ given by $d\psi(z,v)= (\psi(z), d_z\psi(v))$.

\begin{lem}
\label{lem:semicontinuity}
 For any $r\in {\mathbb N}\cup \{0\}$, the set $\{x\in X\mid {\rm dim}(G_x)+ {\rm dim}({\rm ker}(d_1\kappa_x))\geq r\}$ is closed.
\end{lem}

\begin{proof}
 Define $\alpha\colon G\times X\ra X\times X$ by $\alpha(g,x)= (g\cdot x, x)$.  We obtain a morphism $d\alpha$ from $T(G\times X)\iso TG\times TX$ to $T(X\times X)\iso TX\times TX$.  Let $x\in X$ and consider the point $y:= ((x,0), (x,0))\in TX\times TX$.  Now $(d\alpha)^{-1}(y)$ is a closed subset $C_y$ of $TG\times TX$; it is clear that $C_y= \{((g,v),(x,0))\mid g\in G_x, v\in {\rm ker}(d_1\kappa_x))\}$.  Each irreducible component of this set has dimension $e'(y):= {\rm dim}(G_x)+ {\rm dim}({\rm ker}(d_1\kappa_x))$.
 
Define a function $s\colon X\ra TG\times TX$ by $s(x)= ((1,0), (x,0))$.  We identify $X$ with a closed subset of $TX\times TX$ via the embedding $x\mapsto ((x,0), (x,0))$.  Since $s$ is a morphism, we deduce from the upper semi-continuity of the function $e(v)$---taking $(V,W,f)= (TG\times TX, TX\times TX, d\alpha)$---that the function $e'(y)$ is also upper semi-continuous.  The result now follows.
\end{proof}

A morphism $\phi:X \to Y$ of affine varieties is said to be \emph{finite} if the coordinate ring $k[X]$ is integral
over the image of the comorphism $\phi^*:k[Y] \to k[X]$.
Finite morphisms are closed \cite[Prop.\ I.7.3(i)]{Mumfordbook}; in particular, a dominant finite morphism is surjective.
A morphism of affine varieties is called \emph{quasi-finite} if its fibres are finite; finite morphisms are always quasi-finite \cite[Prop.\ I.7.3(ii)]{Mumfordbook},
but the converse is not true.
A dominant morphism $\phi:X\to Y$ of irreducible varieties is called \emph{birational} if the comorphism induces an isomorphism
of function fields $k(X)\iso k(Y)$.
Given an irreducible affine variety $X$, we can form the \emph{normalization} of $X$ by
considering the normal affine variety $\widetilde{X}$ whose coordinate ring is the integral closure of $k[X]$ in the function field $k(X)$.
The normalization map $\nu_X:\widetilde{X} \to X$ is, by construction, finite, birational and surjective.

\begin{rem}
\label{rem:fin_fact}
We record an observation which we use several times in the sequel. Let $\phi:X\to Y$ and $\psi:Y\to Z$ be morphisms of affine varieties with $\psi\circ \phi$ finite. Then it is easy to see that:
\begin{itemize}
\item[(i)] $\phi$ is finite;
\item[(ii)] if $\phi$ is dominant then $\psi$ is finite.
\end{itemize}
\end{rem}

We say that a property $P(x)$ holds for \emph{generic} $x\in X$ if there is an open dense subset $U$ of $X$ such that $P(x)$ holds for all $x\in U$.

\medskip
For the remainder of the paper, we fix the convention that $G$ denotes a {\bf reductive} linear algebraic group over $k$.

\subsection{Group actions and quotients}
\label{sec:catquot}

The main result of this paper, Theorem \ref{thm:mainthm}, concerns quotients of affine varieties by reductive algebraic group actions.
Let $X$ be an affine $G$-variety.
As noted above, $G$ acts on $k[X]$, and we can form the subring $k[X]^G \subseteq k[X]$ of $G$-invariant functions on $X$.
It follows from \cite{nagata2} and \cite{haboush0} that $k[X]^G$ is finitely generated, and hence we can form an affine variety denoted $X/G$ with
coordinate ring $k[X/G] = k[X]^G$.
Moreover, the inclusion $k[X]^G \hookrightarrow k[X]$ gives rise to a morphism from $X$ to $X/G$, which we shall denote by $\pi_{X,G}:X \to X/G$.  The map $\pi_{X,G}$ has the following properties
\cite[Thm.\ A.1.1]{mumford}, \cite[Thm.\ 3.5]{newstead}, \cite[$\S$2]{BaRi}:
\begin{itemize}
\item[(i)] $\pi_{X,G}$ is surjective;
\item[(ii)] $\pi_{X,G}$ is constant on $G$-orbits in $X$;
\item[(iii)] $\pi_{X,G}$ separates disjoint closed $G$-invariant subsets of $X$;
\item[(iv)] each fibre of $\pi_{X,G}$ contains a unique closed $G$-orbit, and $\pi_{X,G}$ determines a bijective map from the set of closed $G$-orbits in $X$ to $X/G$;
\item[(v)] $X/G$ is a {\em categorical quotient} of $X$: that is, for every variety $V$ and every morphism $\psi\colon X\ra V$ which is constant on $G$-orbits, there is a unique morphism $\psi_G\colon X/G\ra V$ such that $\psi= \psi_G\circ \pi_{X,G}$.
\end{itemize}
(This means $\pi_{X,G}$ is a {\em good quotient} in the sense of \cite[Chapter 3, $\S4$, p57]{newstead}.  More generally, if $X$ is a quasi-projective $G$-variety and $\pi$ is a map from $X$ to another quasi-projective variety $Y$ then we call $\pi$ a good quotient if it is an affine map and satisfies (i)--(v) above.)  We say that $\pi_{X,G}\colon X\ra X/G$ is a {\em geometric quotient} if the fibres of $\pi_{X,G}$ are precisely the $G$-orbits.  This is the case if and only if every $G$-orbit is closed (for instance, if every $G$-orbit has the same dimension---e.g., if $G$ is finite).

If $\phi:Y\to X$ is a $G$-equivariant morphism of affine $G$-varieties, then
the restriction of the comorphism to $k[X]^G$ induces a natural morphism from $Y/G$ to $X/G$, which we shall denote by $\phi_G$.
In a special case of this construction, we have the following result, which follows from \cite[Thm.\ 3.5, Lem.\ 3.4.1]{newstead}.  

\begin{lem}
\label{lem:insep}
Let $X$ be an affine $G$-variety and let $i:Y \to X$ be an embedding of a closed $G$-stable subvariety $Y$ in $X$.  Then $\pi_{X,G}(Y)$ is closed in $X/G$.  Moreover, the induced map $i_G\colon Y/G\ra X/G$ is injective and finite.
\end{lem}

\begin{rem}
\label{rem:noiso}
If ${\rm char}(k)= 0$ then $i_G$ is an isomorphism onto its image.  This need not be the case in positive characteristic: see Example~\ref{ex:noslice}.
\end{rem}

We record some other useful results.  First, note that if $G$ is a finite group, then the map $\pi_G$ above is a finite morphism.
To see this, let $f\in k[X]$ and let $T$ be an indeterminate.
Then the polynomial $F(T):=\prod_{g\in G} (T-g\cdot f)\in (k[X])[T]$ is monic and has coefficients in $k[X]^G$,
and $F(f) = 0$.
This shows that $k[X]$ is integral over $k[X]^G$, which gives the claim.

If $X$ is irreducible and normal then $X/G$ is normal \cite[2.19(a)]{BaRi}, while if $G$ is connected then $k[X]^G$ is integrally closed in $k[X]$ \cite[2.4.1]{BaRi}. 

Now suppose $H$ is a subgroup of $G$ such that the normalizer $N_G(H)$ is reductive.
Then the inclusion $X^H \subseteq X$ induces a map of quotients $\psi_{X,H}:X^H/N_G(H) \to X/G$.
Theorem \ref{thm:mainthm} asserts that when $H$ is a $G$-completely reducible subgroup of $G$ (in the sense of Section~\ref{subsec:cr} below),
this map is always a finite morphism.

For technical reasons, we sometimes need to work with affine $G$-varieties satisfying an extra property.  

\begin{defn}
 Let $X$ be an affine $G$-variety.  We denote by $X_{\rm cl}$ the closure of the set $\{x\in X\mid G\cdot x\ \mbox{is closed}\}$.  
 Following Luna \cite[Sec.~4]{luna0}, we say that $X$ has {\em good dimension } (``bonne dimension'') if $X_{\rm cl}= X$.  We say that $x$ is a {\em stable point} of $X$ for the $G$-action if ${\rm dim}(G\cdot x)$ is maximal and $G\cdot x$ is closed \cite[Ch. 3, $\S$4]{newstead}, \cite[Ch. 1, $\S$4]{mumford}.
\end{defn}

\begin{rem}
\label{rem:stableopen}
The set of stable points is open \cite[Ch. 3, $\S$4]{newstead}, \cite[Ch. 1, $\S$4]{mumford} (this is true even without the assumption that $G$ is reductive).  Since the set $\{x\in X\mid \mbox{${\rm dim}(G\cdot x)$ is maximal}\}$ is open, it follows that if $X$ is irreducible then $X$ has good dimension if and only if there exists a stable point.  Moreover, if $X$ has good dimension then generic fibres of $\pi_{X,G}\colon X\ra G$ are orbits of $G$.  Hence if $X$ is irreducible then ${\rm dim}(X/G)= {\rm dim}(X)- m$, where $m$ is the maximal orbit dimension.
\end{rem}

\begin{lem}
\label{lem:fnfldquot}
 Let $X$ be an irreducible affine $G$-variety with good dimension.  Then $k(X/G)= k(X)^G$.  Moreover, $\pi_{X,G}$ is separable.
\end{lem}

\begin{proof}
 It is clear that $k(X/G)$ is a subfield of $k(X)^G$.  Conversely, let $f\in k(X)^G$.  Set 
 $$
 U= \{x\in X\mid \mbox{there exists $h_1,h_2\in k[X]$ such that $f= h_1/h_2$ and $h_2(x)\neq 0$}\}.
 $$  
 Then $U$ is a nonempty open subset of $X$, and clearly $U$ is $G$-stable.  Hence $C:= X\backslash U$ is closed and $G$-stable.  As $X$ has good dimension, there exists $0\neq h\in k[X]^G$ such that $h|_C= 0$.  Now $f$ is a globally defined regular function on the corresponding principal open set $X_h$, so $f\in k[X_h]= k[X][1/h]$.  Hence $\displaystyle f= \frac{f'}{h^r}$ for some $f'\in k[X]$ and some $r\geq 0$.  Then $f'$ is $G$-invariant, since $f$ is, so $f\in k(X/G)$.
 
 The second assertion is \cite[2.1.9(b)]{BaRi}.  Note that separability can fail if $X$ does not have good dimension: see \cite{MaNe}.
\end{proof}

\begin{lem}
\label{lem:finitevis}
 Let $\phi\colon X\ra Y$ be a finite surjective $G$-equivariant map of affine $G$-varieties.
 \begin{itemize}
  \item[(i)] For all $x\in X$, $G\cdot x$ is closed if and only if $G\cdot \phi(x)$ is closed.  Moreover, if $y\in Y$ and $G\cdot y$ is closed then $\phi^{-1}(G\cdot y)$ is a finite union of $G$-orbits, each of which is closed and has the same dimension as $G\cdot y$.
  \item[(ii)] The map $\phi_G\colon X/G\ra Y/G$ is quasi-finite.
  \item[(iii)] $X$ has good dimension if and only if $Y$ does.
 \end{itemize}
\end{lem}

\begin{proof}
 If $x\in X$ and $G\cdot x$ is closed then $G\cdot \phi(x)= \phi(G\cdot x)$ is closed, as $\phi$ is finite. Conversely, let $y\in Y$ such that $G\cdot y$ is closed, and let $n= {\rm dim}(G\cdot y)$.  Let $x\in \phi^{-1}(G\cdot y)$.  Then ${\rm dim}(G_x)\leq {\rm dim}(G_y)$, so ${\rm dim}(G\cdot x)\geq {\rm dim}(G\cdot y)= n$.  But $\phi$ is finite, so every irreducible component of $\phi^{-1}(G\cdot y)$ has dimension $n$.  It follows that ${\rm dim}(G\cdot x)= n$ and $G\cdot x$ is a union of irreducible components of $\phi^{-1}(G\cdot y)$; in particular, $G\cdot x$ is closed.  This proves (i).  Part (iii) now follows.
 
 To prove part (ii), let $x\in X$, $y\in Y$ such that $\phi_G(\pi_{X,G}(x))= \pi_{Y,G}(y)$.  Without loss of generality, we can assume that $G\cdot x$ and $G\cdot y$ are closed.  Now $G\cdot \phi(x)$ is closed by (i), so we must have $G\cdot \phi(x)= G\cdot y$, so $x\in \phi^{-1}(G\cdot y)$.  But $\phi^{-1}(G\cdot y)$ is a finite union of $G$-orbits by (i), so we are done.
\end{proof}

\begin{lem}
\label{lem:finitebirat}
 Let $\phi\colon X\ra Y$ be a finite birational $G$-equivariant morphism of irreducible affine $G$-varieties.  If one of $X$ or $Y$ has good dimension then $\phi_G\colon X/G\ra Y/G$ is birational.
\end{lem}

\begin{proof}
By Lemma~\ref{lem:finitevis}(iii), if one of $X$ or $Y$ has good dimension then they both do.  It follows from Lemma~\ref{lem:fnfldquot} that $k(Y/G)= k(X/G)= k(X)^G$; hence $\phi_G$ is birational.
\end{proof}

Later we also need some material on constructing quotients of projective varieties by actions of reductive groups, but we delay this until Section~\ref{sec:lunaparttwo}.

Suppose $H$ is a subgroup of $G$.
Recall that the quotient $G/H$ (which as a set is just the coset space) has the structure of a
quasi-projective homogeneous $G$-variety,
and $H$ is the stabilizer of the image of $1 \in G$ under the natural map $\pi_{G,H}:G \to G/H$.
Richardson has proved the following in this situation (\cite[Thm.\ A]{rich}; see also \cite{haboush}).

\begin{thm}\label{thm:richquotient}
Suppose $H$ is a subgroup of $G$.
Then $G/H$ is an affine variety if and only if $H$ is reductive.
\end{thm}

Recall also that the Zariski topology on $G/H$ is the quotient topology: that is, a subset $S \subseteq G/H$ is closed in $G/H$
if and only if $\pi_{G,H}\inverse(S)$ is closed in $G$.
We need a technical result.

\begin{lem}
\label{lem:goodmodule}
 Let $H$ be a reductive subgroup of $G$.  There exist a $G$-module $Y$ and a nonempty open subset $U$ of $Y^H$ such that the following hold:
 \begin{itemize}
 \item[(i)] $G_y= H$ for all $y\in U$;
 \item[(ii)] $G\cdot y$ is closed for all $y\in U$;
 \item[(iii)] $N_G(H)\cdot y$ is closed for all $y\in U$.
 \end{itemize}
\end{lem}

\begin{proof}
 Since $H$ is reductive, $G/H$ is affine.  The group $G$ acts on $G/H$ by left multiplication.  Let $x_0= \pi_{G,H}(1)$; then $G_{x_0}= H$.  If $K$ is a reductive subgroup of $G$ containing $H$ then $K\cdot \pi_{G,H}(1)= \pi_{G,H}(K)$ is closed, as $K$ is a closed subset of $G$ that is stable under right multiplication by $H$.  We can embed $G/H$ equivariantly in a $G$-module $X$.  By the lower semi-continuity of orbit dimension, there is a nonempty open subset $U_1$ of $X^H$ such that ${\rm dim}(G_x)= {\rm dim}(H)$ for all $x\in U_1$---so $G_x$ is a finite extension of $H$ for all $x\in U_1$.  If ${\rm char}(k)= 0$ then we can conclude from Proposition \ref{prop:locstab} that there is an open neighbourhood $O$ of $x_0$ such that $G_x\leq H$ for all $x\in O$.  It then follows (applying the arguments for (ii) and (iii) below) that we can take $Y$ to be $X$ and $U$ a suitable nonempty open subset of $X^H\cap O$.  In general, however, we need a slightly more complicated construction.
 
 Let $Y$ be the $G$-module $X\oplus X$.  
 Note that $Y^H= X^H\oplus X^H$ and for any $(y_1,y_2)\in Y^H$, $G_{(y_1,y_2)}= G_{y_1}\cap G_{y_2}$.  We show that $Y$ has the desired properties.  For each $r\geq 0$, define
 $$ C_r= \{y\in U_1\times U_1\mid |G_y:H|\geq r\}. $$
Then $C_r$ is empty for all but finitely many $r$ by \cite[Lem.\ 2.2 and Defn.\ 2.3]{martin3}.  Moreover, $C_r$ is constructible.  For let
$$ \widetilde{C}_r= \{(y,g_1,\ldots, g_r)\mid y\in U_1\times U_1, g_1,\ldots, g_r\in G_y, g_jg_i^{-1}\not\in H\ \mbox{for $1\leq i,j\leq r$}\}; $$
then $C_r$ is the image of $\widetilde{C}_r$ under projection onto the first factor.  Set $D_r= C_r\backslash C_{r+1}$.  Then the nonempty $D_r$ form a finite collection of disjoint constructible sets that cover the irreducible set $U_1\times U_1$, so $D_s$ contains a nonempty open subset $U_2$ of $U_1\times U_1$ for precisely one value of $s$.

We show that $s= 1$.  Suppose not.  Choose $y= (x_1,x_2)\in U_2$.  Let $g_1, g_2, \ldots, g_r$ be coset representatives for $G_{x_1}/H$ with $g_1\in H$.  Note that $U_3= \{x\in U_1\mid (x_1,x)\in U_2\}$ is an open dense subset of $X^H$.  Let $z= (x_1,x)\in U_2$.  Then our hypothesis means that $g\cdot (x_1,x)= (x_1,x)$ for some $g\not\in H$.  Now $g$ must fix $x_1$, so $g\in g_iH$ for some $i\geq 1$; in fact, $i\geq 2$ since $g\not\in H$.  It follows that $g_i$ fixes $(x_1,x)$ since $H$ fixes $(x_1,x)$, so $g_i$ fixes $x$.  But $\bigcup_{j=2}^r (X^{g_j}\cap X^H)$ is a proper closed subset of $X^H$ as none of the $g_j$ for $j\geq 2$ fixes $x_0$, so we have a contradiction.  We conclude that $s= 1$ after all.  Hence $G_y= H$ for all $y\in U_2$.

Set $y_0= (x_0,0)$.  The orbit $N_G(H)\cdot x_0$ is closed in $G/H$, so the orbit $N_G(H)\cdot y_0$ is closed in $Y^H$.  Moreover, $N_G(H)_{y_0}= H$, so $N_G(H)\cdot y_0$ has maximal dimension among the $N_G(H)$-orbits on $Y^H$.  Hence $y_0$ is a stable point of $Y^H$ for the $N_G(H)$-action.  A similar argument shows that $y_0$ is a stable point of $\overline{G\cdot Y^H}$ for the $G$-action.  Since the set of stable points is open in each case, we can find a nonempty open subset $U$ of $U_2$ such that (ii) and (iii) hold for $U$; then (i) holds for $U$ by construction.  This completes the proof.
\end{proof}


\subsection{Cocharacters, $G$-actions and R-parabolic subgroups}
\label{sec:Rpar}

Suppose that $X$ is a $G$-variety.
For any cocharacter $\lambda \in Y(G)$ and $x \in X$
we can define a morphism $\psi = \psi_{x,\lambda}:k^* \to X$
by $\psi(a) = \lambda(a)\cdot x$ for each $a \in k^*$.
We say that the \emph{limit $\lim_{a\to 0} \lambda(a)\cdot x$ exists} if $\psi$
extends to a morphism $\overline{\psi}: k \to X$.
If the limit exists, then the extension $\overline{\psi}$ is unique, and we set
$\lim_{a\to 0} \lambda(a)\cdot x = \overline{\psi}(0)$.
It is clear that, for any $G$ and $X$, if there exists $\lambda \in Y(G)$ such that $\lim_{a\to 0}\lambda(a)\cdot x$ exists
but lies outside $G\cdot x$, then $G\cdot x$ is not closed in $X$.

A subgroup $P$ of $G$ is called a \emph{parabolic subgroup} if the quotient $G/P$ is complete; this is the case if and only if $G/P$ is projective.
If $G$ is connected and reductive, then all parabolic subgroups of $G$ have a \emph{Levi decomposition} $P = R_u(P)\rtimes L$, where the reductive subgroup $L$ is called a \emph{Levi subgroup} of $P$.
In this case, the unipotent radical $R_u(P)$ acts simply transitively on the set of Levi subgroups of $P$,
and given a maximal torus $T$ of $P$ there exists a unique Levi subgroup of $P$ containing $T$.
For these standard results see \cite{borel}, \cite{boreltits} or \cite{springer} for example.
It is possible to extend these ideas to a non-connected reductive group using the formalism of \emph{R-parabolic subgroups}
described in \cite[Sec.\ 6]{BMR}.  We give a brief summary; see {\em loc.\ cit.}\ for further details.
Given a cocharacter $\lambda \in Y(G)$, we have:
\begin{itemize}
\item[(i)] $P_\lambda := \{g \in G \mid \lim_{a\to 0}\lambda(a)g\lambda(a)\inverse \textrm{ exists}\}$
is a parabolic subgroup of $G$; we call a parabolic subgroup arising in this way an \emph{R-parabolic subgroup of $G$}.
\item[(ii)] $L_\lambda := C_G(\lambda) = \{g \in G \mid \lim_{a\to 0}\lambda(a)g\lambda(a)\inverse = g\}$
is a Levi subgroup of $P_\lambda$; we call a Levi subgroup arising in this way an \emph{R-Levi subgroup of $G$}.
\item[(iii)] $R_u(P_\lambda) = \{g \in G \mid \lim_{a\to 0}\lambda(a)g\lambda(a)\inverse = 1\}$.
\end{itemize}
The R-parabolic (resp.\ R-Levi) subgroups of a connected reductive group $G$ are the same as the parabolic and Levi subgroups of $G$.
Moreover, the results listed above for parabolic and Levi subgroups of connected reductive algebraic groups
also hold for R-parabolic and R-Levi subgroups of non-connected reductive groups; that is,
the unipotent radical $R_u(P)$ acts simply transitively on the set of R-Levi subgroups of an R-parabolic subgroup $P$,
and given a maximal torus $T$ of $P$ there exists a unique R-Levi subgroup of $P$ containing $T$.

Now, if $H$ is a reductive subgroup of $G$ and $\lambda \in Y(H)$, then $\lambda$ gives rise in a natural
way to R-parabolic and R-Levi subgroups of both $G$ and $H$.
In such a situation, we reserve the notation $P_\lambda$ (resp.\ $L_\lambda$) for R-parabolic (resp.\ R-Levi) subgroups of $G$, and use the notation
$P_\lambda(H)$, $L_\lambda(H)$, etc. to denote the corresponding subgroups of $H$.
Note that for $\lambda \in Y(H)$, it is obvious from the definitions that $P_\lambda(H) = P_\lambda \cap H$,
$L_\lambda(H) = L_\lambda \cap H$ and $R_u(P_\lambda(H)) = R_u(P_\lambda) \cap H$.

\subsection{$G$-complete reducibility}
\label{subsec:cr}

Our main result, and many of the intermediate ones,
uses the framework of $G$-complete reducibility introduced by J-P. Serre \cite{serre},
which has been shown to have geometric implications in \cite{BMR} and subsequent papers.
We give a short recap of some of the key ideas concerning complete reducibility.

Let $H$ be a subgroup of $G$.
Following Serre (see, for example, \cite{serre}), we say that $H$ is \emph{$G$-completely reducible} ($G$-cr for short)
if whenever $H \subseteq P$ for an R-parabolic subgroup $P$ of $G$, there exists an R-Levi subgroup $L$ of $P$ such that $H \subseteq L$.  For example, if $G= \SL_n(k)$ or $\GL_n(k)$ then $H$ is $G$-cr if and only if the inclusion of $H$ is completely reducible in the usual sense of representation theory.  If $H$ is $G$-cr then $H$ is reductive, while if $H$ is linearly reductive then $H$ is $G$-cr (see \cite[Sec.\ 2.4 and Sec.\ 6]{BMR}).   Hence in characteristic $0$, $H$ is $G$-cr if and only if $H$ is reductive.

In \cite{bateopt} and \cite{martin3} it was shown that the notion of complete reducibility is useful when one considers $G$-varieties and,
as explained in the introduction, one of the purposes of this paper is to expand upon this theme.

The geometric approach to complete reducibility outlined in \cite{BMR}  rests on the following construction,
which was first given in this form in \cite{GIT}.
Given a subgroup $H$ of a reductive group $G$ and a positive integer $n$,
we call a tuple of elements $\mathbf{h} \in H^n$ a \emph{generic tuple} for $H$ if there exists a closed embedding of $G$ in some $\GL_m(k)$ such that $\mathbf{h}$ generates the associative subalgebra of $m\times m$ matrices spanned by $H$ \cite[Defn.\ 5.4]{GIT}.
A generic tuple for $H$ always exists for sufficiently large $n$.
Suppose $\mathbf{h} \in H^n$ is a generic tuple for $H$; then in \cite[Thm.\ 5.8(iii)]{GIT} it is shown that
$H$ is $G$-completely reducible if and only if the $G$-orbit of $\mathbf{h}$ in $G^n$ is closed,
where $G$ acts on $G^n$ by simultaneous conjugation.

\subsection{Optimal cocharacters}
\label{subsec:reductiveaffine}

Let $X$ be an affine $G$-variety.
The classic Hilbert-Mumford Theorem \cite[Thm.\ 1.4]{kempf} says that via the process of taking limits, the cocharacters of $G$ can be used to
detect whether or not the $G$-orbit of a point in $X$ is closed. 
Kempf strengthened the Hilbert-Mumford Theorem in \cite{kempf} (see also \cite{He}, \cite{mumford}, \cite{rousseau}),
by developing a theory of ``optimal cocharacters'' for non-closed $G$-orbits.
We give an amalgam of some results from Kempf's paper; see \cite[Thm.\ 3.4, Cor.\ 3.5]{kempf} (and see also
\cite[$\S$4]{GIT} for the case of non-connected $G$).

\begin{thm}\label{thm:kempf}
Let $x \in X$ be such that $G\cdot x$ is not closed, and let $S$ be a closed $G$-stable subset of $X$
which meets $\overline{G\cdot x}$.
Then there exists an R-parabolic subgroup $P(x)$ of $G$ and a nonempty subset $\Omega(x) \subseteq Y(G)$ such that:
\begin{itemize}
\item[(i)] for all $\lambda \in \Omega(x)$, $\lim_{a\to 0} \lambda(a)\cdot x$ exists, lies in $S$, and is not $G$-conjugate to $x$;
\item[(ii)] for all $\lambda \in \Omega(x)$, $P_\lambda = P(x)$;
\item[(iii)] $R_u(P(x))$ acts simply transitively on $\Omega(x)$;
\item[(iv)] $G_x \subseteq P(x)$.
\end{itemize}
\end{thm}


\section{Preparatory Results}\label{sec:prelims}

In this section we collect some results concerning algebraic group actions on varieties which will be useful in the rest of the paper.
Recall our standing assumption that $G$ is a reductive group.

\subsection{\'Etale slices}
\label{subsec:slices}

\'Etale slices are a powerful tool in geometric invariant theory.  Let $X$ be an affine $G$-variety and let $x\in X$ such that $G\cdot x$ is closed.  Luna introduced the notion of an {\em \'etale slice} through $x$ \cite[III.1]{luna0}: this is a locally closed affine subvariety $S$ of $X$ with $x\in S$ satisfying certain properties.  He proved that an \'etale slice through $x$ always exists when the ground field has characteristic 0.  Bardsley and Richardson later defined \'etale slices in arbitrary characteristic \cite[Defn.\ 7.1]{BaRi} and gave some sufficient conditions for an \'etale slice to exist \cite[Propns.\ 7.3--7.6]{BaRi}.  If an \'etale slice exists through $x$, the orbit $G\cdot x$ must be separable. We record an important consequence of the \'etale slice theory \cite[Prop.\ 8.6]{BaRi}.

\begin{prop}
\label{prop:locstab}
 Let $X$ be an affine $G$-variety and let $x\in X$ such that $G\cdot x$ is closed and there is an \'etale slice through $x$.  Then there is an open neighbourhood $U$ of $x$ such that for all $u\in U$, $G_u$ is conjugate to a subgroup of $G_x$.
\end{prop}

The following example, based on a construction from \cite[Ex.\ 8.3]{martin3}, shows that this result need not hold when there is no \'etale slice.

\begin{ex}
\label{ex:noslice}
 Let $G= \SL_2(k)$ and let $H= C_p\times C_p= \langle \gamma_1,\gamma_2\mid \gamma_1^p= \gamma_2^p= [\gamma_1,\gamma_2]= 1\rangle$.  Define $f\colon k\times H\ra k\times G$ by $f(x,h)= (x,f_x(h))$, where $f_x(\gamma_1^{h_1}\gamma_2^{h_2}):=
 \left(
\begin{array}{cc}
 1 & h_1x+ h_2x^2 \\
 0 & 1
\end{array}
\right)
$.
Set $K_x= {\rm im}(f_x)$.  Note that for each $x\in k$, there are only finitely many $x'\in k$ such that $K_x$ and $K_{x'}$ are $G$-conjugate.  Define actions of $G$ and $H$ on $k\times G$ by $g\cdot (x,g')= (x,gg')$ and $h\cdot (x,g')= (x,g'f_x(h)^{-1})$.  These actions commute with each other, so we have an action of $G$ on the quotient space $V:= (k\times G)/H$.  Set $\varphi= \pi_{k\times G,H}$.   Since $H$ is finite, $\varphi$ is a geometric quotient.  A straightforward calculation shows that for any $(x,g)\in k\times G$, the stabilizer $G_{\varphi(x,g)}$ is precisely $gK_xg^{-1}$.  It follows that the $G$-orbits on $V$ are all closed, but the assertion of Proposition~\ref{prop:locstab} cannot hold for {\em any} $v\in V$.  Hence no $v\in V$ admits an \'etale slice.  Note that generic stabilizers are nontrivial, but there do exist orbits with trivial stabilizer (take $x= 0$).

Nonetheless we can even show (using \'etale slice methods!) that generic $G$-orbits in $V$ are separable.  Let $O= \{x\in k\mid x^2\neq 0, x, \ldots, (p-1)x\}$, an open subset of $k$.  Then the finite group $H$ acts freely on $O\times G$, so by \cite[Prop.\ 8.2]{BaRi}, $O\times G$ is a principal $H$-bundle in the \'etale topology in the sense of \cite[Defn.\ 8.1]{BaRi}.  Let $x\in O$.  It follows that the derivative $d_{(x,g)}\varphi$ is surjective for all $g\in G$.  Define an $H$-equivariant map $\psi_x\colon G\ra k\times G$ by $\psi_x(g)= (x,g)$.  An easy computation shows that the map $(\psi_x)_H\colon G/H\ra V$ induced by $\psi_x$ is bijective and separable when regarded as a map onto its image, so $(\psi_x)_H$ gives by Zariski's Main Theorem an isomorphism from $G/H$ onto its image.  Now $(\psi_x)_H$ is $G$-equivariant, where we let $G$ act on $G/H$ by left multiplication.  Since $\pi_{G,H}\colon G\ra G/H$ is separable, the orbit $G\cdot \pi_{G,H}(g)$ is separable for any $g\in G$.  This means that the orbit $G\cdot \varphi(x,g)= G\cdot (\psi_x)_H(\pi_{G,H}(g))$ is separable as well.

In contrast, consider the orbit $G\cdot \varphi(0,g)$.  This cannot be separable: for otherwise $\varphi(0,g)$ admits an \'etale slice by \cite[Prop.\ 7.6]{BaRi}, since the stabilizer $G_{\varphi(0,g)}$ is trivial, and we know already that this is impossible.  It follows easily that $(\psi_0)_H\colon G/H\ra V$ is not an isomorphism onto its image.  We see from this that if $i$ is the obvious inclusion of $Y:= \{0\}\times G$ in $k\times G$ then the induced map $i_G\colon Y/H\ra (k\times G)/H= V$ is not an isomorphism onto its image (cf.\ Remark~\ref{rem:noiso}).
\end{ex}

The failure of Proposition~\ref{prop:locstab} and other consequences of the machinery of \'etale slices when slices do not exist is behind many of the technical difficulties we need to overcome in order to prove Theorem~\ref{thm:mainthm}.

\subsection{Some results on closed orbits}
\label{subsec:horbits}

We first need a technical lemma which collects together various properties
of orbits and quotients and the associated morphisms.
For more details, see the proofs of \cite[Lem.\ 4.2, Lem.\ 10.1.3]{rich1} or the discussion in \cite[Sec.\ 2.1]{jantzen}, for example; the extension to non-connected $G$ is immediate.  Note that if $G$ acts on a variety $X$ then for any $x\in X$, $G\cdot x$ is locally closed \cite[Prop.~1.8]{borel}, so it has the structure of a quasi-affine variety.

\begin{lem}
\label{lem:quotientorbit}
Let $X$ be a $G$-variety. Suppose $x \in X$, and let $\psi_x:G/G_x \to G\cdot x$ be the natural map.
Then:
\begin{itemize}
\item[(i)] $\psi_x$ is a homeomorphism;
\item[(ii)] $G\cdot x$ is affine if and only if $G/G_x$ is affine if and only if $G_x$ is reductive;
\item[(iii)] $\psi_x$ is an isomorphism of varieties if and only if the orbit $G\cdot x$ is separable.
\end{itemize}
\end{lem}

\begin{rem}
All the subtleties here are only really important in positive characteristic since in characteristic $0$ the orbit map is always separable,
so the morphism $\psi_x$ is always an isomorphism.
The result shows that even in bad cases where the orbit map is not separable we can reasonably compare the quotient $G/G_x$ with the orbit
$G\cdot x$, as one might hope.
\end{rem}

\begin{lem}
\label{lem:closedhorbits}
Let $H$ be a subgroup of $G$ and suppose $x \in X$.
Set $K = G_x$ and let $H$ act on $X$ by restriction of the $G$-action.
Then:
\begin{itemize}
\item[(i)] $H\cdot x$ is closed in $G\cdot x$ if and only if $HK = \{hk\mid h\in H,k\in K\}$ is a closed subset of $G$.
\item[(ii)] If $G\cdot x$ is closed in $X$ then $H\cdot x$ is closed in $X$ if and only if $HK$ is closed in $G$.
\end{itemize}
\end{lem}

\begin{proof}
Part (ii) follows immediately from part (i). For part (i), since the map $\psi_x:G/K \to G\cdot x$ from
Lemma \ref{lem:quotientorbit}
is a homeomorphism, $H\cdot x$ is closed in $G\cdot x$ if and only if the
corresponding subset $H\cdot\pi_{G,K}(1)$ is closed in $G/K$ (recall that $\pi_{G,K}\colon G\ra G/K$ is the canonical projection).
Since $G/K$ has the quotient topology, this is the case if and only if
the preimage of this orbit is closed in $G$.
But the preimage is precisely the subset $HK$.
\end{proof}

Our next result involves the following set-up:
Suppose $Y$ is another $G$-variety.
Then $G\times G$ acts on the product $X\times Y$ via $(g_1,g_2)\cdot(x,y) = (g_1\cdot x,g_2\cdot y)$,
and identifying $G$ with its 
diagonal embedding $\Delta(G)$ in $G\times G$, we can also get the diagonal action
of $G$ on $X\times Y$: $g\cdot(x,y) = (g\cdot x,g\cdot y)$.

\begin{lem}\label{lem:twovarieties}
With the notation just introduced, let $x \in X$, $y \in Y$ and set $K = G_x$, $H=G_y$. Then:
\begin{itemize}
\item[(i)] $H\cdot x$ is closed in $G\cdot x$ if and only if $K\cdot y$ is closed in $G\cdot y$ if and only if $G\cdot(x,y)$ is
closed in $(G\cdot x)\times (G\cdot y)$.
\item[(ii)] If $G\cdot x$ is closed in $X$ and $G\cdot y$ is closed in $Y$, then
$H\cdot x$ is closed in $X$ if and only if $K\cdot y$ is closed in $Y$ if and only if $G\cdot(x,y)$ is closed in $X \times Y$.
\end{itemize}
\end{lem}

\begin{proof}
(i). The first equivalence follows from Lemma \ref{lem:closedhorbits} since $KH = (HK)^{-1}$ is closed in $G$ if and only if $HK$ is closed in $G$
(note that this argument is based on the one in the proof of \cite[Lem.\ 10.1.4]{rich1}).
For the second equivalence, consider the orbit map $\kappa_1:G\times G \to G$ associated to the orbit of $1\in G$ for the double coset action of $G\times G$ on $G$ (cf.\ Section~\ref{sec:dblecosets}); so $\kappa_1$ is given by $\kappa_1(g_1,g_2) = g_1g_2^{-1}$. 
Then $\kappa_1$ is surjective and open.
Now, since the $(G\times G)$-orbit of $(x,y)$ is $(G\cdot x) \times (G\cdot y)$ and the stabilizer of $(x,y)$ in $G\times G$ is $K \times H$, we have that
$G\cdot (x,y)= \Delta(G)\cdot(x,y)$ is closed in $(G\cdot x) \times (G\cdot y)$ if and only if $\Delta(G)(K\times H)$ is closed in $G\times G$,
by Lemma \ref{lem:closedhorbits}(i).
Now $\Delta(G)(K\times H)$ is closed in $G\times G$ if and only if $(K\times H)\Delta(G)$ is, and $(K\times H)\Delta(G) = \kappa_1^{-1}(KH)$.
Since $\kappa_1$ is a surjective open map, we conclude that $\Delta(G)(K\times H)$ is closed in $G\times G$ if and only if $KH$ is closed in $G$,
which happens if and only if $K\cdot y$ is closed in $G\cdot y$, by Lemma \ref{lem:closedhorbits}(i) again.

(ii). This chain of equivalences follows quickly from part (i).
\end{proof}

\begin{rem}
The results above give criteria for a result of the form ``$G\cdot x$ closed implies $H\cdot x$ closed'' for a point $x$ in a $G$-variety $X$.
We can't hope for a general converse to this. 
For example, let $G$ be any connected reductive group and, in the language of Section~\ref{subsec:cr},
let $x \in X=G^n$ be a generic tuple for a Borel subgroup of $G$ and $y \in Y=G^n$ be a generic tuple for $G$ itself.
Then, $G_x = G_y = Z(G)$, the $G$-orbits of $y$ and $(x,y)$ are closed, but the $G$-orbit of $x$ is not closed.
\end{rem}

\subsection{Finite morphisms and quotients}
\label{sec:benquotientsarefinite}
In this section we provide some general results on finite morphisms and 
quotients by reductive group actions.
We begin with an extension of Zariski's Main Theorem which deals with nonseparable morphisms.  
Recall that if $X$ is an irreducible affine variety then $\nu_X:\widetilde{X}\to X$ denotes the normalization of $X$.

\begin{prop}
\label{prop:insep_ZMT}
 Let $\phi\colon X\ra Y$ be a dominant quasi-finite morphism of irreducible affine varieties.
 Suppose $Y$ is normal and generic fibres of $\phi$ are singletons.
 Then $\phi$ is a finite bijection onto an open subvariety of $Y$.
 Moreover, the normalization map $\nu_X\colon \widetilde{X}\ra X$ is a bijection.
\end{prop}

\begin{proof}
 As $\phi$ is dominant, we may identify $k[Y]$ with a subring of $k[X]$ and $k(Y)$ with a subfield of $k(X)$.
 The hypothesis on the fibres of $\phi$ implies that $\phi$ is purely inseparable \cite[Thm.\ 4.6]{Hum}.
 Let $f_1,\ldots, f_r$ be generators for $k[X]$ as a $k$-algebra.
 Then there exists a power $q$ of $p$ such that $f_i^q\in k(Y)$ for all $i$.
 Let $S$ be the $k$-algebra generated by $k[Y]$ together with $f_1^q,\ldots, f_r^q$ and let $Z$ be the corresponding affine variety, so that $S=k[Z]$.
 Then the inclusions $k[Y]\subseteq k[Z]\subseteq k[X]$ give rise to maps $\psi\colon X\ra Z$ and $\alpha\colon Z\ra Y$ such that $\phi= \alpha\circ \psi$.
 Now $k[X]$ is integral over $k[Z]$ by construction, so $\psi$ is finite and surjective, and hence $\alpha$ is quasi-finite and has the same image as $\phi$.
 But $\alpha$ is birational by construction, so $\alpha$ is an isomorphism from the affine variety $Z$ onto an open subvariety of $Y$ by Zariski's Main Theorem (since $Y$ is normal).
 To complete the proof of the first assertion, it is enough to show that $\psi$ is injective.
 This follows because any $k$-algebra homomorphism $k[X]\ra k$ is completely determined by its values on $f_1^q,\ldots, f_r^q$, which are elements of $k[Z]$.

 Because $\nu_X$ is finite and birational, the map $\phi\circ\nu_X\colon \widetilde{X}\ra Y$ satisfies the hypotheses of the proposition.
 Hence $\phi\circ\nu_X$ is injective.  This forces $\nu_X$ to be injective also.
 But $\nu_X$ is also surjective, and we are done.
\end{proof}

We need some further results about the behaviour of affine $G$-varieties under normalization.
If $X$ is an affine $G$-variety then $\widetilde{X}$ inherits a unique structure of a $G$-variety
such that $\nu_X$ is $G$-equivariant (cf.\ \cite[Sec.\ 3]{BaRi}).
This gives a map of quotients $(\nu_X)_G\colon \widetilde{X}/G \to X/G$.

\begin{lem}
\label{lem:normquot}
 Let $X$ be an irreducible affine $G$-variety with good dimension and let $(\nu_X)_G$ be as above. Then $(\nu_X)_G$ is finite and $\widetilde{X}/G$ is the normalization of $X/G$.
\end{lem}

\begin{proof}
The natural map of quotients $X/G^0 \to X/G$ can be viewed as the quotient map by the finite group $G/G^0$ and is therefore finite.
The same is true for $\widetilde{X}/G^0 \to \widetilde{X}/G$, so by Remark~\ref{rem:fin_fact}(ii) it follows that
$(\nu_X)_G$ is finite if $(\nu_X)_{G^0}$ is. Hence we may assume that $G$ is connected.

The coordinate ring $k[\widetilde{X}]$ of the normalization of $X$ is the integral closure of $k[X]$ in the function field $k(X)$.
Let $S$ be the integral closure of $k[X]^G$ in $k(X)$.  
Then $S$ is finitely generated as a $k$-algebra \cite[2.4.3]{BaRi}, 
and $S\subseteq k[\widetilde{X}]$ as $k[\widetilde{X}]$ is integrally closed, so $S\subseteq k[\widetilde{X}]^G$ as $G$ is connected (see the proof of \cite[2.4.1]{BaRi}).
Let $Z$ be the affine variety corresponding to $S$.  
Then $(\nu_X)_G$ factors as $\widetilde{X}/G\stackrel{\alpha}{\ra} Z\stackrel{\beta}{\ra} X/G$.  
It is clear that $Z$ is normal (in fact, $S$ is the integral closure of $k[X]^G$ in $k(X)^G$, so $Z$ is the normalization of $X/G$).  
Now $(\nu_X)_G$ is birational and quasi-finite by Lemmas~\ref{lem:finitebirat} and \ref{lem:finitevis}(ii), so $\alpha$ is also birational and quasi-finite.  
It follows from Zariski's Main Theorem that $\alpha$ is an open embedding.

The map $\beta$ is finite by construction, so to complete the proof that $(\nu_X)_G$ is finite it is enough to show that $\alpha$ is surjective.
Define $\theta\colon \widetilde{X}\ra Z\times X$ by $\theta= (\alpha\circ \pi_{\widetilde{X},G})\times \nu_X$ and let $C$ be the closure of $\theta(\widetilde{X})$.  We have a commutative diagram

$$\xymatrixcolsep{5pc}\xymatrix{
\widetilde{X}\ar[d]_{\pi_{\widetilde{X},G}}\ar[r]^{\theta}& C\ar[d]^{{\rm pr_1}}\\
\widetilde{X}/G\ar[r]^{\alpha}& Z\\}$$\\

\noindent where ${\rm pr}_1$ is projection onto the first factor.  The composition $\widetilde{X}\ra C\ra X$ is finite, where the second map is projection onto the second factor, so $\theta$ is a finite map from $\widetilde{X}$ to $C$; in particular, $C= \theta(\widetilde{X})$.

 Let $G$ act on $Z\times X$ trivially on the first factor, and by the given action on the second.  It is immediate that $\theta$ is $G$-equivariant, so $C$ is $G$-stable and we have an induced map $\theta_G\colon \widetilde{X}/G\ra (Z\times X)/G$.  The image $D$ of $\theta_G$ is $\pi_{Z\times X,G}(C)$, and this is closed in $(Z\times X)/G$ as $C$ is closed and $G$-stable.  There is an obvious map $\xi\colon (Z\times X)/G\ra Z\times X/G$, and it is easily checked that $\xi$ is an isomorphism; hence $\xi(D)$ is closed.  Untangling the definitions, we find that $\alpha$ factors as $\widetilde{X}/G\stackrel{\theta_G}{\ra} (Z\times X)/G\stackrel{\xi}{\ra} Z\times X/G\stackrel{\tau}{\ra} Z$, where $\tau$ is projection onto the first factor.
 
 Clearly $\xi(D)$ is contained in the subset $\{(z,e)\in Z\times X/G\mid \beta(z)= e\}$, which we can identify with $Z$ via $\tau$.  It follows that $\alpha(\widetilde{X}/G)= \tau(\xi(D))$ is closed in $Z$.  But $\alpha(\widetilde{X}/G)$ is a nonempty open subset of $Z$, so we must have $\alpha(\widetilde{X}/G)= Z$, as required.
 
 To finish the proof, we note that for any $G$ (connected or otherwise), the variety $\widetilde{X}/G$ is normal since $\widetilde{X}$ is normal, and the considerations above show that
$(\nu_X)_G:\widetilde{X}/G \to X/G$ is finite.
Moreover,  $(\nu_X)_G$ is birational by Lemma~\ref{lem:finitebirat} since $X$ has good dimension.
The result now follows from another application of Zariski's Main Theorem.
\end{proof}

Next we extend a result of Bardsley and Richardson \cite[2.4.2]{BaRi}, which they prove in the special case when $X$ and $Y$ are normal and $\phi$ is dominant.  It provides an extension to positive characteristic of a result used freely in \cite{luna}.

\begin{prop}
\label{prop:finquotfin}
 Let $\phi\colon X\ra Y$ be a finite $G$-equivariant morphism of affine $G$-varieties.  Then $\phi_G\colon X/G\ra Y/G$ is finite.
\end{prop}

\begin{proof}
As at the start of the proof of Lemma~\ref{lem:normquot}, we can immediately reduce to the case when $G$ is connected, since the natural maps $X/G^0\ra X/G$ and $Y/G^0 \to Y/G$ are finite.   
The map $X_{\rm cl}/G\ra X/G$ is surjective, and Lemma~\ref{lem:insep} implies it is finite.  
We may also assume, therefore, that $X$ has good dimension.  
Since a morphism is finite if and only if its restriction to every irreducible component of the domain is finite, we can assume $X$ is irreducible.   
By the proof of Lemma~\ref{lem:finitevis}, $\phi(X_{\rm cl})\subseteq Y_{\rm cl}$, so after replacing $Y$ with $\overline{\phi(X)}$ if necessary, we may assume by Lemma~\ref{lem:insep} that 
$\phi$ is dominant and $Y$ is irreducible and has good dimension.
 
The map $\phi\colon X\ra Y$ gives rise to a map $\widetilde{\phi}\colon \widetilde{X}\ra \widetilde{Y}$, and $\widetilde{\phi}$ is finite as $\phi$ is.  
We have a commutative diagram 

$$\xymatrixcolsep{5pc}\xymatrix{
\widetilde{X}\ar[d]_{\nu_{X}}\ar[r]^{\widetilde{\phi}}& \widetilde{Y}\ar[d]^{\nu_{Y}}\\
X\ar[r]^{\phi}&Y\\}$$\\

\noindent where the vertical arrows are the normalization maps.  
Taking quotients by $G$, we obtain a commutative diagram 

$$\xymatrixcolsep{5pc}\xymatrix{
\widetilde{X}/G\ar[d]_{(\nu_X)_G}\ar[r]^{\widetilde{\phi}_G}& \widetilde{Y}/G\ar[d]^{(\nu_Y)_G}\\
X/G\ar[r]^{{\phi}_G}&Y/G\\}$$

Since $\widetilde{\phi}$ is finite and dominant and $\widetilde{X}$ and $\widetilde{Y}$ are irreducible and normal, the map $\widetilde{\phi}_G:\widetilde{X}/G\ra \widetilde{Y}/G$ is finite and dominant \cite[2.4.2]{BaRi}.  
Now Lemma~\ref{lem:normquot} shows that the map $(\nu_Y)_G:\widetilde{Y}/G \to Y/G$ is finite and so $(\nu_Y)_G\circ \widetilde{\phi}_G$ is finite. Therefore, $\phi_G\circ (\nu_X)_G$ is finite and by Remark~\ref{rem:fin_fact}(ii) we get that $\phi_G$ is finite, as required.
\end{proof}

\section{Proof of Theorem \ref{thm:mainthm}, Part 1: quasi-finiteness}
\label{sec:lunapartone}

In this section we provide the first step towards our proof of Theorem \ref{thm:mainthm}, showing that the map
$\psi_{X,H}$ in question is quasi-finite.
We are also able to retrieve other results from \cite{luna} which follow from the main theorem,
but in arbitrary characteristic.
Our first result is a generalization of \cite[Thm.\ 4.4]{bateopt}; see also \cite[Thm.\ 5.4]{BMRT}.

\begin{prop}
\label{prop:htog}
Suppose that $G$ is a reductive group and $X$ is an affine $G$-variety.
Let $H$ be a $G$-completely reducible subgroup of $G$ and let $x\in X^H$.
Then the following are equivalent:
\begin{itemize}
\item[(i)] $N_G(H)\cdot x$ is closed in $X$;
\item[(ii)] $G\cdot x$ is closed in $X$ and $H$ is $G_x$-cr.
\end{itemize}
\end{prop}

\begin{proof}
First suppose $G\cdot x$ is not closed.
Let $P(x)$ and $\Omega(x)$ be the R-parabolic subgroup and class of cocharacters given by Theorem~\ref{thm:kempf}.
Since $H \leq G_x \leq P(x)$ is $G$-cr, there exists an R-Levi subgroup $L$ of $P(x)$ containing $H$.
Since $R_u(P(x))$ acts simply transitively on $\Omega(x)$ and on the set of R-Levi subgroups of $P(x)$,
there exists $\lambda \in \Omega(x)$ with $L = L_\lambda$.
But then $H \subseteq L_\lambda$ means that $\lambda \in Y(C_G(H)) \subseteq Y(N_G(H))$;
in particular, $\lambda(a)\cdot x \in N_G(H)\cdot x$ for all $a \in k^*$.
Now $\lim_{a\to 0} \lambda(a)\cdot x$ exists in $X$ and is not $G$-conjugate to $x$,
so it is not $N_G(H)$-conjugate to $x$, so $N_G(H)\cdot x$ is not closed.
This shows that if (i) holds then $G\cdot x$ must be closed.
Therefore, in order to finish the proof, we need to show that $N_G(H)\cdot x$ is closed if and only if
$H$ is $G_x$-cr under the assumption that $G\cdot x$ is closed
(note that since $G\cdot x$ is closed, $G_x$ is reductive (Lemma~\ref{lem:quotientorbit}(ii)),
and hence it makes sense to ask whether or not $H$ is $G_x$-cr).

To see this equivalence, let $\mathbf{h} \in G^n$ for some $n$ be a generic tuple for the subgroup $H$ and consider the diagonal action of $G$ on $G^n\times X$.
Then $C_G(H) = G_{\mathbf{h}}$.
Now, by Lemma \ref{lem:twovarieties},
since $G\cdot x$ is closed in $X$, $C_G(H)\cdot x$ is closed in $X$
if and only if $G_x\cdot\mathbf{h}$ is closed in $G^n$.
The latter condition is equivalent to requiring that $H$ is $G_x$-cr,
and since $x$ is $H$-fixed and $N_G(H)$ is a finite extension of $HC_G(H)$, $C_G(H)\cdot x$ is closed in $X$ if and only if $N_G(H)\cdot x$ is closed in $X$.
This completes the proof.
\end{proof}

\begin{rems}
\label{rems:failure}
(i). In characteristic $0$, the subgroup $H$ of $G$ is $G$-cr if and only if $H$ is reductive.
In this case, therefore, we are just requiring that $H$ is reductive and the second condition
in part (ii) of the Theorem is then automatic.
Therefore, when $\Char(k) = 0$, we retrieve Luna's result \cite[$\S$3, Cor.\ 1]{luna}.

(ii). The implication (ii) implies (i) of Proposition \ref{prop:htog} is not true in general without the hypothesis that $H$ is $G_x$-cr, as a straightforward modification of
\cite[Ex.\ 4.6]{bateopt} shows.  See also \cite[Ex.\ 5.1, Ex.\ 5.3]{BMR2}, noting that if $A,B$ are commuting $G$-cr subgroups of $G$ and $B$ is not $C_G(A)$-cr then $B$ is not $N_G(A)$-cr by \cite[Prop.\ 2.8]{BMR2}.

(iii) Suppose $H$ is a torus in Proposition~\ref{prop:htog}; then $H$ is linearly reductive, so $H$ is $G$-cr.  Now $N_G(H)$ is a finite extension of the Levi subgroup $C_G(H)$ of $G$, so $N_G(H)\cdot x$ is closed if and only if $C_G(H)\cdot x$ is closed.  Moreover, $H$ is automatically $G_x$-cr if $G_x$ is reductive.  It follows that $G\cdot x$ is closed if and only if $C_G(H)\cdot x$ is closed.  (This is also a special case of \cite[Thm.~5.4]{cochars}.)  
We use this result repeatedly in Section~\ref{sec:dblecosets}.
\end{rems}

Some of the constructions used in the proof of the next result are based on those in \cite[Sec. 3.8]{relative}.

\begin{lem}
\label{lem:notGcrfixedpoint}
Suppose $H$ is a reductive subgroup of $G$ such that $H$ is not $G$-cr.
Then:
\begin{itemize}
\item[(i)] There exists an affine $G$-variety $X$ and a point $x \in X^H$ such that $G\cdot x$ is not closed.
\item[(ii)] There exists a rational $G$-module $V$ and a nonzero subspace $W\subseteq V^H$ such that:
\begin{itemize}
\item[(a)] $0$ lies in the closure of $G\cdot w$ for all $w\in W$;
\item[(b)] $N_G(H)\cdot w$ is finite (hence closed in $V$) for all $w\in W$.
\end{itemize}
In particular, if $N_G(H)$ is reductive, then the map $\psi_{V,H}:V^H/N_G(H) \to V/G$ is not quasi-finite.
\end{itemize}
\end{lem}

\begin{proof}
Choose a closed embedding $G \hookrightarrow \SL_m(k)$ for some $m$ and think of $H$ and $G$ as closed subgroups of $\SL_m(k)$.
Let $\Mat_m$ denote the algebra of all $m\times m$ matrices.
Let $\mathbf{x} = (x_1,\ldots,x_n) \in H^n$ be a basis for the associative subalgebra of $\Mat_m$ spanned by $H$;
then $\mathbf{x}$ is a generic tuple for $H$ (see Section \ref{subsec:cr}).
This means that if we let $\SL_m(k)$ act on $Y := (\Mat_m)^n$ by simultaneous conjugation, then $G\cdot \mathbf{x}$ is not closed.
Note that since $H$ is itself $H$-cr, the $H$-orbit of $H\cdot \mathbf{x}$ is closed in $Y$.

There is also a right action of $\GL_n(k)$ on $Y$, which we denote by $\ast$.
Given a matrix $A=(a_{ij}) \in \GL_n(k)$ and an element $\mathbf{y} = (y_1,\ldots,y_n) \in Y$, we can set
$$
\mathbf{y}\ast A = \left(\sum_{i=1}^n a_{i1}y_i, \ldots, \sum_{i=1}^n a_{in}y_i\right).
$$
This is the action obtained by thinking of the tuple $\mathbf{y}$ as a row vector of length $n$ and letting the $n\times n$ matrix $A$
act on the right in the obvious way.
Note that the $\SL_m(k)$- and $\GL_n(k)$-actions commute.

Given any $h \in H$, since $\mathbf{x}$ is a basis for the associative algebra generated by $H$, we have that $h\cdot\mathbf{x}$
is also a basis for this algebra, and hence there exists a unique $A(h) \in \GL_n(k)$ such that $h \cdot \mathbf{x} = \mathbf{x} \ast A(h)$.
Note also that $$
(h_1h_2) \cdot \mathbf{x} = h_1\cdot (h_2\cdot \mathbf{x}) = h_1\cdot (\mathbf{x}\ast A(h_2)) = (h_1\cdot \mathbf{x})\ast A(h_2) = \mathbf{x} \ast (A(h_1)A(h_2)),
$$
and hence the map $A:H \to \GL_n(k)$ is a group homomorphism.
This map is in fact a rational representation of $H$ since it arises from the morphic
action of $H$ on the vector space spanned by the entries of $\mathbf{x}$.
Let $K$ denote the image of $H$ in $\GL_n(k)$; then $K$ is a reductive group and $\mathbf{x}\ast K= H\cdot \mathbf{x}$ is closed.
Moreover, since the elements of the tuple $\mathbf{x}$ are linearly independent, the stabilizer of $\mathbf{x}$ in $K$ is trivial.
Hence $\mathbf{x}$ is a stable point for the action of $K$ on $Y$.
Now let $X= Y/K$
and set $x:=\pi_{Y,K}(\mathbf{x})$.
Since the $\SL_m(k)$- and $\GL_n(k)$-actions on $Y$ commute, we obtain an action of $\SL_m(k)$ on $X$.
It is immediate that $x \in X^H$.

We know that $G\cdot \mathbf{x}$ is not closed in $Y$, so there exists a cocharacter $\lambda \in Y(G)$ such that
$\lim_{a\to 0} \lambda(a) \cdot \mathbf{x} = \mathbf{y}$ exists and is not $G$-conjugate to $\mathbf{x}$.
Since $\pi_{Y,K}$ is $G$-equivariant, it is easy to see that $\lim_{a\to 0} \lambda(a) \cdot x = \pi_{Y,K}(\mathbf{y})$
(and in particular this limit exists).
Suppose $\pi_{Y,K}(\mathbf{y})$ is $G$-conjugate to $x$.
Then there exists $g \in G$ such that $g\cdot\pi_{Y,K}(\mathbf{y}) = \pi_{Y,K}(g\cdot\mathbf{y}) = x$,
so $g\cdot\mathbf{y} \in \pi_{Y,K}^{-1}(x) = \pi_{Y,K}^{-1}(\pi_{Y,K}(\mathbf{x}))$.
But $\mathbf{x}$ is a stable point for $K$, so $\pi_{Y,K}\inverse(\pi_{Y,K}(\mathbf{x}))$ is
precisely $K\cdot \mathbf{x}$, which coincides with $H\cdot \mathbf{x}$ by construction.
Hence $g\cdot\mathbf{y} = h\cdot\mathbf{x}$ for some $h \in H$ and we see that $\mathbf{y}$ and $\mathbf{x}$ are $G$-conjugate, which is
a contradiction.
Hence $\pi(\mathbf{y})$ and $x$ are not conjugate, and the $G$-orbit of $x \in X^H$
is not closed, which proves (i).

To prove (ii), let $S$ denote the unique closed $G$-orbit in the closure of $G\cdot x$.
Then, following \cite[Lemma 1.1(b)]{kempf},
we can find a rational $G$-module $V$ with a $G$-equivariant morphism $\phi:X\to V$
such that $\phi^{-1}(0) = S$.
Since $G\cdot x$ is not closed, it does not meet $S$, and hence $v:=\phi(x) \neq 0$.
However, by Theorem~\ref{thm:kempf},
there exists $\mu \in Y(G)$ such that $\lim_{a\to 0} \mu(a)\cdot x \in S$, and since the
morphism $\phi$ is $G$-equivariant, we have that $\{0\}$ is the unique closed $G$-orbit in the
closure of $G\cdot v$.
Note also that $v$ is $H$-fixed since $x$ is.
Now the tuple $\mathbf{x}$ consists of elements of $H$, so is $C_G(H)$-fixed,
and hence $x = \pi_{Y,K}(\mathbf{x})$ is also $C_G(H)$-fixed, which means that $x$ is actually $HC_G(H)$-fixed.
Since $H$ is reductive, $N_G(H)^0 = H^0C_G(H)^0$, so $x$ is
$N_G(H)^0$-fixed and hence the $N_G(H)$-orbit of $x$ is finite.
This in turn implies that the $N_G(H)$-orbit of $v$ is finite, and hence closed in $V$.
Finally, let $W\subseteq V^H$ be the one-dimensional subspace of $V$ spanned by $v$.
Then for all $w \in W$, $N_G(H)\cdot w$ is finite, hence closed, and $0$ is in the closure of $G\cdot w$,
so we have parts (a) and (b) of (ii).
For the second statement, if $N_G(H)$ is reductive---so that it definitely makes sense to talk about the
quotient $V^H/N_G(H)$---then the image of $W$ in $V^H/N_G(H)$ is still infinite, but every element
of this infinite set is mapped to the point corresponding to $0$ in $V/G$ under the natural morphism
$V^H/N_G(H) \to V/G$, so this morphism cannot be quasi-finite.
\end{proof}

With this result in hand, we can
provide the first step towards the proof of Theorem \ref{thm:mainthm}
by showing that the morphism $\psi_{X,H}$ is quasi-finite.

\begin{thm}
\label{thm:quasifinite}
Suppose $H$ is a reductive subgroup of $G$.
The following conditions on $H$ are equivalent:
\begin{itemize}
\item[(i)] $N_G(H)$ is reductive and for every affine $G$-variety $X$, the natural
morphism $\psi_{X,H}:X^H/N_G(H)\to X/G$ is quasi-finite;
\item[(ii)] $H$ is $G$-cr.
\end{itemize}
\end{thm}

\begin{proof}
Suppose $H$ is not $G$-cr.
Then either $N_G(H)$ is not reductive, in which case the first part of condition (i) fails, or else $N_G(H)$ is reductive but the
second part of condition (i) fails by Lemma~\ref{lem:notGcrfixedpoint}(ii).
Hence (i) implies (ii).

Conversely, suppose $H$ is $G$-cr, and let $X$ be any affine $G$-variety.
Since $H$ is $G$-cr, and hence $H$ is reductive, we have $N_G(H)^0 = H^0C_G(H)^0$.
That $N_G(H)$ is reductive is shown in \cite[Prop. 3.12]{BMR},
and hence it always makes sense to take the quotient $X^H/N_G(H)$.

Suppose $x\in X^H$.
We first claim that the unique closed $G$-orbit $S$ in $\overline{G\cdot x}$ meets $X^H$.
Indeed, either $G\cdot x$ is already closed, in which case $S=G\cdot x$, or we can find the optimal parabolic $P(x)$ and
optimal class $\Omega(x)$ as given in Kempf's Theorem \ref{thm:kempf}.
Since $H\leq G_x \leq P(x)$ and $H$ is $G$-cr, there is
a Levi subgroup $L$ of $P(x)$ containing $H$.
Since the unipotent radical acts simply transitively on $\Omega(x)$ and on the set of Levi subgroups of $P(x)$,
there is precisely one element
$\lambda \in \Omega(x)$ with $L = L_\lambda$, and this choice of $\lambda$ commutes with $H$.
But then $y:= \lim_{a\to0}\lambda(a)\cdot x \in S\cap X^H$, which proves the claim.

Now any point of $X/G$ has the form $\pi_{X,G}(x)$, where $G\cdot x$ is closed in $X$.  So let $x\in X$ such that $G\cdot x$ is closed.
For any $y\in \pi_{X,G}^{-1}(\pi_{X,G}(x)) \cap X^H$, $G\cdot x$ is the unique closed $G$-orbit in $\overline{G\cdot y}$.
Hence, if $\pi_{X,G}^{-1}(\pi_{X,G}(x)) \cap X^H$ is nonempty, $G\cdot x$ must meet $X^H$, by the claim in the previous paragraph.
It follows from the definitions that $\pi_{X^H,N_G(H)}^{-1}(\psi_{X,H}^{-1}(\pi_{X,G}(x))) = \pi_{X,G}^{-1}(\pi_{X,G}(x))\cap X^H$,
so to show that $\psi_{X,H}$ is quasi-finite,
we need to show that for each such $x$ there are only finitely many closed $N_G(H)$-orbits in $\pi_{X,G}^{-1}(\pi_{X,G}(x))\cap X^H$.
But any $y\in X^H$ with a closed $N_G(H)$-orbit has a closed $G$-orbit, by Proposition \ref{prop:htog},
and hence any $y\in \pi_{X,G}^{-1}(\pi_{X,G}(x))\cap X^H$ with a closed $N_G(H)$-orbit is already $G$-conjugate to $x$.
So we must show that there are only finitely many closed $N_G(H)$-orbits in $G\cdot x\cap X^H$.

Fix $x\in X^H$ with $G\cdot x$ closed, and recall that $G_x$ is reductive since $G\cdot x $ is closed.
Let $y\in G\cdot x\cap X^H$, and write $y=g\cdot x$ for some $g\in G$.
Since $G\cdot y$ is closed, Proposition \ref{prop:htog} says that $N_G(H)\cdot y$ is closed if and only if $H$ is $G_y$-cr,
which is the case if and only if $g\inverse Hg$ is $G_x$-cr.
Suppose $g\inverse Hg$ and $H$ are $G_x$-conjugate: say $H = g_1\inverse (g\inverse Hg)g_1$ for some $g_1\in G_x$.
Then $gg_1 \in N_G(H)$ and $y = g\cdot x = (gg_1)\cdot x$, so we see that $x$ and $y$ are $N_G(H)$-conjugate.  Conversely, suppose $x$ and $y$ are $N_G(H)$-conjugate: say $y= m\cdot x$ for some $m\in N_G(H)$.  Then $m^{-1}g\in G_x$ and $m^{-1}g(g^{-1}Hg)g^{-1}m= H$, so $g\inverse Hg$ and $H$ are $G_x$-conjugate.
Hence the distinct closed $N_G(H)$-orbits in $G\cdot x\cap X^H$ correspond to the distinct
$G_x$-conjugacy classes of $G_x$-cr subgroups of the form $g\inverse Hg$ inside $G_x$.  It is therefore enough to show that there are only finitely many such conjugacy classes.

Let $\mathbf{h} \in H^n$ be a generic tuple for $H$ in $G_x$ for some $n$ and let $g\in G$ such that $g\inverse Hg$ is a $G_x$-cr subgroup of $G_x$.  Then $g\inverse\cdot\mathbf{h}$ is a generic tuple for $g\inverse H g$.  
Since $g\inverse Hg$ is both $G$-cr and $G_x$-cr, 
 the $G$- and $G_x$-orbits of $\mathbf{h}$ in $G^n$ are both closed.
It follows from \cite[Thm.\ 1.1]{martin1} that the natural map of quotients $G_x^n/G_x \to G^n/G$ is finite,
and hence there are only finitely many closed $G_x$-orbits contained in $G\cdot\mathbf{h} \cap G_x^n$.  This proves the result.
\end{proof}

\begin{rem}
\label{rem:singletons}
Note that if $G_x = H$ and $G\cdot x$ is closed then the argument in the proof above shows that there is precisely one closed $N_G(H)$-orbit inside $G\cdot x \cap X^H$ (namely, $N_G(H)\cdot x$),
and therefore $\psi_{X,H}^{-1}(\pi_{X,G}(x))$ is a singleton.  We will use this observation in Sections~\ref{sec:lunapartthree} and \ref{sec:sep}.
\end{rem}

The third paragraph of the proof above shows that for any $x \in X^H$, the unique closed orbit contained
in $\overline{G\cdot x}$ also meets $X^H$.
This allows us to prove the following:

\begin{lem}
\label{lem:closedcrit}
The map $\psi_{X,H}\colon X^H/N_G(H)\ra X/G$ of Theorem \ref{thm:mainthm}
has closed image if and only if for all
$x\in \overline{G\cdot X^H}$ such that $G\cdot x$ is closed, $x\in G\cdot X^H$.
\end{lem}

\begin{proof}
Since $\overline{G\cdot X^H}$ is closed and $G$-stable, we may replace $X$ with 
$\overline{G\cdot X^H}$; 
then saying $\psi_{X,H}$ has closed image is the same as saying that $\psi_{X,H}$ is surjective.
But this is equivalent to saying that the fibre above every point of $X/G$ meets $X^H$.
Since each fibre contains a unique closed orbit, the observation before the Lemma gives the result.
\end{proof}

Now we extend Luna's result \cite[Cor.\ 3]{luna} to positive characteristic.

\begin{prop}
\label{prop:luna2}
Suppose $H$ is a reductive subgroup of $G$.
The following conditions on $H$ are equivalent:
\begin{itemize}
\item[(i)] for every affine $G$-variety $X$, every $G$-orbit in $X$ that meets $X^H$ is closed;
\item[(ii)] $H$ is $G$-cr and $N_G(H)/H$ is a finite group.
\end{itemize}
\end{prop}

\begin{proof}
Suppose (i) holds.
Then $H$ must be $G$-cr, by Lemma~\ref{lem:notGcrfixedpoint}.
Since $H$ is reductive, $N_G(H)^0 = H^0C_G(H)^0$.
Let $x \in C_G(H)^0$ and let $G$ act on itself by conjugation.
We have $x \in C_G(H) = G^H$, so the $G$-orbit of $x$ (i.e., the conjugacy class of $x$) must be closed in $G$.
As $x$ belongs to $G^0$, it follows from \cite[Cor.\ 3.6]{SS} that $x$ is a semisimple element of $G$.
Since $C_G(H)^0$ consists entirely of semisimple elements, it must be a torus \cite[Cor.\ 11.5(1)]{borel}.
Hence $N_G(H)^0 = H^0C_G(H)^0$ is a reductive group and $(N_G(H)/H)^0$ is a torus.

Now suppose, for contradiction, that $N_G(H)/H$ is infinite.
Then there exists a one-dimensional subtorus $S$ of $C_G(H)^0$ not contained in $H$.
To ease notation, let $Z = HS$ and note that $Z$ is reductive.
Since $H$ is normal in $Z$ and $Z/H \iso S/(S\cap H)$ is a one-dimensional torus, we have a
multiplicative character $\chi:Z \to k^*$ with kernel $H$;
let $V$ denote the corresponding $1$-dimensional $Z$-module.
Set $Y = G\times V$, let $Z$ act on $Y$ via $z\cdot(g,v) := (gz\inverse,\chi(z)v)$,
and let $G$ act by left multiplication on the first factor and trivially on the second factor.
Now let $X = Y/Z$; this is a special case of a construction described in \cite[I.3]{luna0}.  Since $Z$ is reductive and $Y$ is affine, $X$ is affine, and since $Z$ acts freely on $Y$, the fibres of $\pi_{Y,Z}$ are precisely the $Z$-orbits in $Y$.
Moreover, since the $G$- and $Z$-actions on $Y$ commute, $X$ is naturally a $G$-variety.  Let $0\neq v\in V$ and choose a cocharacter $\lambda$ of $Z$ such that $m:= -\langle \lambda,\chi\rangle> 0$.  Then $\lambda(a)\cdot \pi_{Y,Z}(1,v)= \pi_{Y,Z}(1,\chi(\lambda(a^{-1}))v)= \pi_{Y,Z}(1,a^mv)$ for all $a\in k^*$, so $\lim_{a\to 0} \lambda(a)\cdot \pi_{Y,Z}(1,v)= \pi_{Y,Z}(1,0)\not\in G\cdot \pi_{Y,Z}(1,v)$, so $G\cdot \pi_{Y,Z}(1,v)$ is not closed.  However, $\pi_{Y,Z}(1,v)$ is $H$-fixed, and we have our contradiction.  Hence $N_G(H)/H$ is finite. This completes the proof that (i) implies (ii).

Conversely, suppose (ii) holds and $X$ is any affine $G$-variety.
Let $x \in X^H$, so that $H \leq G_x$.
Since $N_G(H)/H$ is finite, $N_G(H)\cdot x$ is a finite union of $H$-orbits.
But $N_G(H)\cdot x \subseteq X^H$, so each of these $H$-orbits is a singleton and $N_G(H)\cdot x$ is finite, and therefore closed in $X$.
Now we can apply Proposition \ref{prop:htog} to deduce that the $G$-orbit of $x$ is also closed, which gives (i).
\end{proof}


\section{Proof of Theorem \ref{thm:mainthm}, Part 2: surjectivity}
\label{sec:lunaparttwo}

In this section, we prove the following:

\begin{thm}
\label{thm:closed}
 Let $X$ be an affine $G$-variety and let $H$ be a $G$-cr subgroup of $G$.
 Then the map $\psi_{X,H}\colon X^H/N_G(H)\ra X/G$ has closed image.
\end{thm}

The proof of Theorem~\ref{thm:closed} in positive characteristic requires some preparation.  Before we begin, we note that if ${\rm char}(k)= 0$ then we can give a much quicker proof using the machinery of \'etale slices, as follows.  Let $x\in \overline{G\cdot X^H}$ such that $G\cdot x$ is closed.  Then there is an \'etale slice through $x$ for the $G$-action \cite[III.1]{luna0}.  By Proposition~\ref{prop:locstab}, there is an open $G$-stable neighbourhood $O$ of $x$ such that $G_y$ is conjugate to a subgroup of $G_x$ for all $y\in O$.  Since $O$ meets $G\cdot X^H$, $H$ must be conjugate to a subgroup of $G_x$.  Hence $x\in G\cdot X^H$, and we are done by Lemma~\ref{lem:closedcrit}.

We need some material on weighted projective varieties and their quotients by reductive groups (cf.\ \cite[Chapter 3, $\S4$]{newstead}).  Let $V$ be a $G$-module equipped with an action of $k^*$ which commutes with the action of $G$. 
Suppose the weights of $k^*$ on $V$ are all positive, so that $\lim_{c\to0} c\cdot v = 0$ for every $v\in V$, where the $c\in k^*$.
The action of $k^*$ decomposes $V$ into weight spaces, and this in turn gives a grading  by non-negative integers of the coordinate ring $k[V]$.  Let $k[V]_i$ denote the $i^{\rm th}$-graded piece of $k[V]$.  We say that $f\in k[V]$ is {\em homogeneous} if $f\in k[V]_i$ for some $i$; in this case we write ${\rm deg}(f)= i$ (so ${\rm deg}(f)$ is the weighted degree rather than the usual degree of a polynomial).  The action of $k^*$ can be diagonalised, so we can choose a basis $\{v_1,\ldots, v_n\}$ for $V$ consisting of weight vectors.  Then the corresponding elements $X_1,\ldots, X_n$ of the dual $V^*$ are weight vectors and we can write $k[V] = k[X_1,\ldots,X_n]$; we set $d_i= {\rm deg}(X_i)$ for each $i$.

Set ${\mathbb W}(V) = \mathrm{Proj}(k[V])$; we call this the \emph{weighted projectivization} of $V$ according to the $k^*$-action \cite{dolgachev}.  Then ${\mathbb W}(V)$ is a projective variety and we may identify the points of ${\mathbb W}(V)$ with the equivalence classes of $V\setminus \{0\}$ under the equivalence relation $\sim$, where $v\sim w$ if and only if $v= c\cdot w$ for some $c\in k^*$.  If the weights of the $k^*$-action on $V$ are all 1---that is, if the action of $k^*$ on $V$ is by ordinary scalar multiplication---then ${\mathbb W}(V)$ is just the ordinary projective space ${\mathbb P}(V)$ associated to $V$, but in Section~\ref{sec:lunapartthree} we will need to consider the general weighted case.  One can show that the canonical projection $\xi_V\colon V\backslash \{0\}\ra {\mathbb W}(V)$ is a good quotient.  If $f\in k[V]$ is homogeneous and ${\rm deg}(f)\geq 1$ then we set ${\mathbb W}(V)_f= \{\xi_V(v)\mid v\in V, f(v)\neq 0\}$; then ${\mathbb W}(V)_f$ is an open affine subset of ${\mathbb W}(V)$, with coordinate ring $(k[V]_f)_0$ (the zero-graded part of the localisation $k[V]_f$).

Since the $G$- and $k^*$-actions commute, the ring $k[V]^G$ of invariants also inherits a grading by non-negative integers: if $f \in k[V]^G$ and $f = f_0 + \cdots +f_r$ is a decomposition with $f_i \in k[V]_i$ for each $i$, then $f_i \in (k[V]^G)_i$ for each $i$.  It is easily checked that the action of $G$ on $V$ descends to give an action of $G$ on ${\mathbb W}(V)$.  We say that $x\in {\mathbb W}(V)$ is a {\em semistable point} (or {\em $G$-semistable point}) if $x\in {\mathbb W}(V)_f$ for some homogeneous $f\in k[V]^G$ such that ${\rm deg}(f)\geq 1$; otherwise we say that $x$ is {\em unstable} (or {\em $G$-unstable}).  We define ${\mathbb W}(V)_{{\rm ss}, G}$ to be the set of $G$-semistable points of ${\mathbb W}(V)$; this is an open subset of ${\mathbb W}(V)$.

Let $Y= {\rm Proj}(k[V]^G)$.  Then $Y$ is a projective variety and the inclusion of $k[V]^G$ in $k[V]$ gives rise to a map $\eta_{V,G}\colon {\mathbb W}(V)_{{\rm ss},G}\ra Y$.  It follows from the proof of \cite[Thm.\ 3.14]{newstead} that $Y$ is a good quotient of ${\mathbb W}(V)_{{\rm ss},G}$ in the sense of \cite[Chapter 3, $\S4$, p57]{newstead} (the argument given in {\em loc.\ cit.}\ is only for the ordinary projective variety ${\mathbb P}(V)$, but it is clear that it holds for the weighted case as well).  We set ${\mathbb W}(V)_{{\rm ss}, G}/G:= Y$.  Moreover, if $f\in k[V]^G$ is homogeneous and ${\rm deg}(f)\geq 1$ then $Y_f:= \eta_{V,G}({\mathbb W}(V)_f)$ is an open affine subvariety of $Y$, with coordinate ring $((k[V]^G)_f)_0$, and the induced map of affine varieties from ${\mathbb W}(V)_f$ to $Y_f$ is a good quotient.

We have an analogous notion of semistable points in the affine variety $V$.  We say that $v\in V$ {\em semistable} (or {\em $G$-semistable}) if $f(v)\neq 0$ for some homogeneous $f\in k[V]^G$ such that ${\rm deg}(f)\geq 1$, and we define $V_{{\rm ss},G}$ to be the set of semistable points; note that $V_{{\rm ss},G}= \xi_V^{-1}({\mathbb W}(V)_{{\rm ss},G})$.  If $v$ is not stable then we say that $v$ is {\em unstable} (or {\em $G$-unstable}).  Since the homogeneous elements of $k[V]^G$ generate $k[V]^G$, $v$ is unstable if and only if $\pi_{V,G}(v)= \pi_{V,G}(0)$.  By the Hilbert-Mumford Theorem, this is the case if and only if there exists $\lambda\in Y(G)$ such that $\lim_{a\to 0} \lambda(a)\cdot v= 0$.  We denote the composition $V_{{\rm ss},G}\to {\mathbb W}(V)_{{\rm ss},G}\stackrel{\eta_{V,G}}{\longrightarrow} Y$ by $\nu_{V,G}$.

Now suppose $K$ is a reductive subgroup of $G$ and $X$ is a closed 
$(K\times k^*)$-stable subvariety of $V$ (so that in particular $0 \in X$).
Then the vanishing ideal for $X$ in $k[V]$ is homogeneous with respect to our fixed $k^*$-grading, so $k[X]$ inherits a grading.  The constructions above still go through replacing $V$ and $G$ with $X$ and $K$.  We have projective varieties ${\mathbb W}(X):= {\rm Proj}(k[X])$ and ${\mathbb W}(X)_{{\rm ss}, K}/K:= Z:= {\rm Proj}(k[X]^K)$, where ${\mathbb W}(X)_{{\rm ss},K}$ is defined analogously to above; the map ${\mathbb W}(X)_{{\rm ss},K}\ra Z$ is a good quotient.  Note that the proof of \cite[Thm.\ 3.14]{newstead} still goes through: all one needs is that $k[X]$ is graded and the $G$-action preserves the grading.  Since ${\mathbb W}(V)$ and ${\mathbb W}(X)$ are categorical quotients of $V\setminus \{0\}$ and $X\setminus \{0\}$ respectively, the inclusion of $X$ in $V$ gives rise to a map from ${\mathbb W}(X)$ to ${\mathbb W}(V)$.

It is clear from the characterisation of semistable points in terms of the Hilbert-Mumford Theorem that $X\cap V_{{\rm ss},G}\subseteq X_{{\rm ss},K}$.  Suppose $X_{{\rm ss},K}\subseteq V_{{\rm ss},G}$; then $X_{{\rm ss},K}= X\cap V_{{\rm ss},G}$.  Since $Y$ and $Z$ are categorical quotients of ${\mathbb W}(V)_{{\rm ss},G}$ and ${\mathbb W}(X)_{{\rm ss},K}$, respectively, the inclusion of ${\mathbb W}(X)_{{\rm ss},K}$ in ${\mathbb W}(V)_{{\rm ss},G}$ gives rise to a map $\phi\colon Z\ra Y$.  Now we come to the point: because $Y$ and $Z$ are projective, the image of $\phi$ is closed.

We can now state and prove the main result of this section.

\begin{prop}\label{prop:closedstrong}
 Let $V$ be a $G$-module equipped with a $k^*$-action as above.  Let $K$ be a reductive subgroup of $G$, let $X$ be a closed $(K\times k^*)$-stable subset of $V$ and suppose $X_{{\rm ss},K}\subseteq V_{{\rm ss},G}$.  Then the natural morphism of quotients $X/K \to V/G$
has closed image (i.e., $\pi_{V,G}(X)$ is closed in $V/G$).
\end{prop}

\begin{proof}
For the purposes of the proof, we need to replace $G$ with a slightly larger 
group
to take into account possible effects of passing to the weighted projectivisation.
Without loss, we may assume that $G$ is a subgroup of $\GL(V)$.
Let $R = k[V]^G$ and let $f_1,\ldots, f_r$ be homogeneous generators for $R$.
Let $m$ be the lowest common multiple of the degrees $\deg(f_1),\ldots,\deg(f_r)$ and write $m = p^\alpha m'$ for some $m'$ coprime to $p$.
Let $F$ be the finite group of $m'^{\rm th}$ roots of unity, regarded as a subgroup of $k^*$ (equipped with its given action on $V$).
Now set $\Gamma = FG$.
Then $\Gamma$ inherits an action on $V$ from the commuting actions of $G$ and $k^*$,
and $V_{{\rm ss},\Gamma} = V_{{\rm ss},G}$ because $\Gamma^0 = G^0$.
Further, $F$ acts on the quotient $V/G$ and the quotient map $\pi_{V/G,F}:V/G \to (V/G)/F = V/\Gamma$ is a geometric quotient.
The subset $X$ of $V$ is $k^*$-stable, and hence $F$-stable, so $\pi_{V,G}(X)$ is a  
$\pi_{V/G,F}$-saturated subset of $V/G$ -- that is, $\pi_{V/G,F}^{-1}(\pi_{V/G,F}(\pi_{V,G}(X))) = \pi_{V,G}(X)$.
Hence, to show the result claimed, we may replace $G$ with $\Gamma$ and show that $\pi_{V,\Gamma}(X)$ is closed in $V/\Gamma$.
Now let $S = k[V]^\Gamma \subseteq k[V]^G$.
A homogeneous $f \in R$ belongs to $S$ if and only if $\deg(f)$ is divisible by $m'$ (since then the action of $F$ is killed by the degree).

To show that $\pi_{V,\Gamma}(X)$ is closed in $V/\Gamma$, it is enough to show that for every $x \in \overline{\Gamma\cdot X}$ with closed $\Gamma$-orbit,
there exists an $x' \in X$ with $\pi_{V,\Gamma}(x') = \pi_{V,\Gamma}(x)$
(cf.\ Lemma \ref{lem:closedcrit}).  
If $x \in \overline{\Gamma\cdot X}$ is unstable and $\Gamma\cdot x$ is closed, then $x$ must actually be $0$, so $x \in X$ also. 
Therefore, we may assume that we have $x \in V_{{\rm ss},\Gamma} \cap \overline{\Gamma\cdot X}$ such that $\Gamma\cdot x$ is closed (in $V$).
By the discussion before the proposition, we have a morphism $\phi:\mathbb{W}(X)_{{\rm ss},K}/K \to \mathbb{W}(V)_{{\rm ss},\Gamma}/\Gamma$ with closed image $C$, say.
Note that since we are assuming $X_{{\rm ss},K} = V_{{\rm ss},\Gamma}\cap X$,
we have $\Gamma\cdot X_{{\rm ss},K} = V_{{\rm ss},\Gamma}\cap \Gamma\cdot X$,
and this set is dense in $V_{{\rm ss},\Gamma}\cap \overline{\Gamma\cdot X}$.
The composition 
$V_{{\rm ss},\Gamma}\stackrel{\xi_V}{\lra} {\mathbb W}(V)_{{\rm ss},\Gamma}\stackrel{\eta_{V,\Gamma}}{\lra} {\mathbb W}(V)_{{\rm ss},\Gamma}/\Gamma$
takes $\Gamma\cdot X_{{\rm ss},K}$ into $C$, 
and since $C$ is closed that means that the 
composition in fact takes all of $V_{{\rm ss},\Gamma} \cap \overline{\Gamma\cdot X}$ into $C$.
Therefore, we can find $z$ in $\mathbb{W}(X)_{{\rm ss},K}/K$ with $\phi(z) = \eta_{V,\Gamma}(\xi_V(x))$.
Tracing back through the definitions, we see that $\phi(z) = \eta_{V,\Gamma}(\xi_V(y))$ for some $y \in X_{{\rm ss},K}$.
It follows that $\eta_{V,\Gamma}(\xi_V(x))=\eta_{V,\Gamma}(\xi_V(y))$;
we claim that in fact $\pi_{V,\Gamma}(x) = \pi_{V,\Gamma}(c\cdot y)$ for some 
$c \in k^*$. 
Note that suffices to finish the proof, since in particular $y \in X$, so setting $x'=c\cdot y \in X$ gives us what we want.

Both points $x$ and $y$ lie in $V_{{\rm ss},\Gamma}$, so there are homogeneous generators $f_i,f_j\in R$ for which $f_i(x)\neq 0$ and $f_j(y) \neq 0$.
By definition of $m$, there are $m_i,m_j\in \NN$ such that $f_i^{m_i}$ and $f_j^{m_j}$ both have degree $m$.
Taking a suitable linear combination of $f_i^{m_i}$ and $f_j^{m_j}$, we can therefore find a homogeneous $f \in R$ of degree $m$ for which $f(x)\neq0\neq f(y)$.
Now we can choose $c\in k^*$ such that $f(x) = f(c\cdot y)$.

Let $f' \in S$ be non-constant and homogeneous; as previously observed, $f' \in S$ means that $\deg(f') = rm'$ for some $r \in \NN$. Then $(f')^{p^\alpha}$ has degree $rm$, so $\frac{(f')^{p^\alpha}}{f^r}$ has degree $0$ in the localization $R_f$.
Further, since $x$ and $c\cdot y$ have the same image in 
$(\mathbb{W}(V)_{{\rm ss},\Gamma}/\Gamma)_f$, we have $\frac{(f')^{p^\alpha}}{f^r}(x) = \frac{(f')^{p^\alpha}}{f^r}(c\cdot y)$.
Since $f(x) = f(c\cdot y)\neq 0$, this in turn implies that
$(f')^{p^\alpha}(x) = f'^{p^\alpha}(c\cdot y)$, and hence $f'(x) = f'(c\cdot y)$.
Since $S$ is generated by homogeneous elements, we see that $\pi_{V,\Gamma}(x) = \pi_{V,\Gamma}(c\cdot y)$, as required. This finishes the proof. 
\end{proof}

\begin{proof}[Proof of Theorem~\ref{thm:closed}]
  We can choose a $G$-equivariant closed embedding of $X$ in a $G$-module $V$.  Let $v\in V^H$ such that $v$ is $G$-unstable.  By the argument in the proof of Proposition \ref{prop:htog}, there exists $\lambda\in Y(N_G(H))$ such that $\lim_{a\to 0} \lambda(a)\cdot v= 0$.
  This shows that $(V^H)_{{\rm ss},N_G(H)}\subseteq V_{{\rm ss},G}$.
  The $G$-action commutes with the natural $k^*$-action by scalars, and this preserves the subspace $V^H$ also,
  so Proposition \ref{prop:closedstrong} implies that the map $\psi_{V,H}:V^H/N_G(H)\to V/G$ has closed image.
  
  Now $\overline{G\cdot X^H}\cap V^H\subseteq X\cap V^H= X^H$, so $\overline{G\cdot X^H}\cap V^H= X^H$.  
  Let $x\in \overline{G\cdot X^H}$ such that $G\cdot x$ is closed. Then, since $x\in \overline{G\cdot V^H}$ and $\psi_{V,H}$ has closed image, Lemma \ref{lem:closedcrit} implies $x\in G\cdot V^H$, so we can write $x=g\cdot v$ for some $v\in V^H$.
  But then $v \in \overline{G\cdot X^H}\cap V^H =X^H$, so $x\in G\cdot X^H$ and we are done by Lemma \ref{lem:closedcrit}.
  \end{proof}

\begin{rem}
For the proof of Theorem \ref{thm:closed}, we only need to apply Proposition~\ref{prop:closedstrong} when the $k^*$-action is the standard action by scalars, so the weighted projectivization is the usual projectivization in this case. However, we do need Proposition~\ref{prop:closedstrong} in this more general set-up to complete the proof of Theorem~\ref{thm:mainthm} in the next section.
\end{rem}

\begin{ex}
 Let $G$ act on $X:= G$ by conjugation and let $H$ be a maximal torus of $G$.  Assume $G$ is connected.  Then $X^H= H$.  Since the closed orbits in $X$ are precisely the semisimple conjugacy classes \cite{SS}, the map
 $\psi_{X,H}\colon X^H/N_G(H)\ra X/G$ is surjective---in fact, it is well known that $\psi_{X,H}$ is an isomorphism (cf.\ Section~\ref{sec:sep}).  Note, however, that although $G\cdot X^H$ is dense in $X$, not every element of $X$ belongs to $G\cdot H$ (just take $x\in X$ not semisimple).
\end{ex}


\section{Proof of Theorem \ref{thm:mainthm}, Part 3: finiteness}
\label{sec:lunapartthree}

We now complete the proof of Theorem~\ref{thm:mainthm}.  The implication (ii)$\implies$(i) follows from Theorem~\ref{thm:quasifinite}, so it remains to show that if $H$ is $G$-completely reducible then the morphism $\psi_{X,H}$ is finite.
By Lemma~\ref{lem:insep}, we can replace $X$ with a larger affine $G$-variety, hence without loss we can assume that $X$ is a $G$-module.
 
  Let $G_1$ be the subgroup of $G$ generated by $G^0$ and $H$.  The inclusion of $X^H$ in $X$ gives rise to a morphism $\psi^1_{X,H}\colon X^H/N_{G_1}(H)\ra X/G_1$.  We have a commutative diagram
 
 \begin{equation}
 \label{eqn:G1diag}
   \xymatrixcolsep{5pc}\xymatrix{
X^H/N_{G_1}(H)\ar[d]_{}\ar[r]^{\psi^1_{X,H}}& X/G_1\ar[d]^{}\\
X^H/N_G(H)\ar[r]^{\psi_{X,H}}&X/G\\}
 \end{equation}

 \noindent where the vertical arrows are the obvious maps.  We may identify $X/G$ with the quotient of $X/G^0$ by the finite group $G/G^0$, so the map $X/G^0\ra X/G$ is finite.  This map factorizes as $X/G^0\ra X/G_1\ra X/G$, so the map $X/G_1\ra X/G$ is finite by Remark~\ref{rem:fin_fact}(ii).  Likewise, the map $X^H/N_{G_1}(H)\ra X^H/N_G(H)$ is finite.  Hence both of the vertical maps in (\ref{eqn:G1diag}) are finite and surjective. Now using Remark~\ref{rem:fin_fact}(ii) we see it is enough to show that $\psi^1_{X,H}$ is finite. So it is enough to prove that $\psi_{X,H}$ is finite under the assumption that $G= G_1$.

 Let $Y$ and $U\subseteq Y^H$ be as in Lemma~\ref{lem:goodmodule} and set $V= X\oplus Y$.  We have a $G$-equivariant closed embedding of $X$ in $V$ given by $x\mapsto (x,0)$.  Let $W= \overline{G\cdot V^H}$; then $W^H= V^H$.  Note that $G\cdot V^H= G^0\cdot V^H$ by our assumption that $G= G_1$, so $W$ is irreducible.  By Lemma~\ref{lem:insep} again, it is enough to prove that $\psi_{W,H}$ is finite.

 The subset $X^H\times U$ of $X^H\oplus Y^H= V^H= W^H$ is open and dense, and $G_w= H$ for $w\in X^H\times U$.  Next we claim that $W$ has good dimension (for the $G$-action).  To see this, let $y_0\in U$ and set $w_0= (0,y_0)$.  Then $G\cdot w_0$ is a $G$-orbit of maximal dimension in $W$, and $G\cdot w_0$ is closed (as $G\cdot y_0$ is, by Lemma~\ref{lem:goodmodule}), so $w_0$ is a stable point of $W$ for the $G$-action.  The claim now follows from Remark~\ref{rem:stableopen}.  By a similar argument, $W^H$ has good dimension for the $N_G(H)$-action.
 Now since the stable points form an open subset, we can conclude that $G\cdot w$ and $N_G(H)\cdot w$ are closed for generic $w\in W^H$, and it follows from Remark~\ref{rem:singletons} that generic fibres of $\psi_{W,H}$ are singletons.

 Now consider the normalization $\widetilde{W}$ of $W$. 
 Since the normalization map $\nu_W\colon \widetilde{W}\to W$ is birational, 
 $\widetilde{W}$ contains an open dense subvariety $\widetilde{O}$ such that the map $\widetilde{O}\to W$ 
 is an isomorphism onto its image $O$, and $O$ is open in $W$.   
 We can take $\widetilde{O}$ and $O$ to be $G$-stable, so the latter meets $G\cdot W^H$.  
 Now $G\cdot W^H$ is constructible and dense in $W$, so it contains a nonempty open subset of $W$.  
 Hence, by adjusting $O$ and $\widetilde{O}$ if necessary, we can assume that $O\subseteq G\cdot W^H$
 and $O$ is $N_G(H)$-stable.  
 Since $\nu_W$ is $G$-equivariant, we get an isomorphism from $\widetilde{O}^H$ onto $O^H$.  
 Let $C$ be the closure in $\widetilde{W}$ of $\widetilde{O}^H$; then $C\subseteq \widetilde{W}^H$ and $C$ is $N_G(H)$-stable.
 Further, since the open subset $\widetilde{O}^H$ in $C$ is isomorphic to the open subset $O^H$ in the irreducible set $W^H$, $C$ is irreducible; it follows from Zariski's Main Theorem that $C$ is isomorphic to $W^H$.
 Hence $C\iso W^H=V^H$ carries a vector space structure, and we can identify a point $0_C\in C$ corresponding to the zero $0_V$; we have $\nu_W(0_C)= 0_V$ by construction.
 Furthermore, the action of $k^*$ on $W$ by scalar multiplication lifts to an action of $k^*$ on $\widetilde{W}$ which preserves the closed subset $C$.
 
 We want to apply Proposition \ref{prop:closedstrong} to deduce that the map $C/N_G(H) \to \widetilde{W}/G$ has closed image (note that we cannot use Theorem~\ref{thm:closed} directly because $C$ might be properly contained in $\widetilde{W}^H$).
 In order to do this, we choose a $(G\times k^*)$-equivariant embedding $i$ of $\widetilde{W}$ in a vector space $M$ such that $0_C$ maps to the zero $0_M\in M$.  (For instance, choose $f_1,\ldots, f_s\in k[\widetilde{W}]$ for some $s$ such that the $f_i$ generate $k[\widetilde{W}]$ as a $k$-algebra and $f_1(0_C)=\cdots = f_s(0_C)= 0$; we can take $M$ to be the dual of $N$, where $N$ is a $(G\times k^*)$-stable subspace of $k[\widetilde{W}]$ containing all the $f_i$.)  Replacing $M$ with the subspace spanned by $i(\widetilde{W})$, we can assume that $i(\widetilde{W})$ spans $M$.
 
 Let $\lambda_0\colon k^*\ra k^*$ be the identity cocharacter of $k^*$.  Now $\{0_V\}$ is the unique closed $k^*$-orbit in $W$, and each element of $W$ is destabilized to $0_V$ by $\lambda_0$.  It follows from Lemma~\ref{lem:finitevis}(i) that $\{0_C\}$ is the unique closed $k^*$-orbit in $\widetilde{W}$.  Let $0_C\neq \widetilde{w}\in \widetilde{W}$.  The Hilbert-Mumford Theorem implies that $\lim_{a\to 0} \lambda_0(a)\cdot \widetilde{w}= 0_C$ or $\lim_{a\to 0} (-\lambda_0)(a)\cdot \widetilde{w}= 0_C$.  In particular, $k^*$ does not fix $\widetilde{w}$, so $k^*$ does not fix $\nu_W(\widetilde{w})$, so $\nu_W(\widetilde{w})\neq 0_V$. Suppose $\lim_{a\to 0} (-\lambda_0)(a)\cdot \widetilde{w}= 0_C$.  Then $\lim_{a\to 0} (-\lambda_0)(a)\cdot \nu_W(\widetilde{w})= \nu_W(0_C)= 0_V$.  But this is impossible because $\lim_{a\to 0} \lambda_0(a)\cdot \nu_W(\widetilde{w})= 0_V$ and $\nu_W(\widetilde{w})\neq 0_V$.  We deduce that $\lim_{a\to 0} \lambda_0(a)\cdot \widetilde{w}= 0_C$.  Hence, we can conclude that $k^*$ acts on $M$ with positive weights.
  Further,  $C$ is a closed $(N_G(H)\times k^*)$-stable subset of $M$. 
 By the same argument as in the proof of Proposition \ref{prop:htog}, if $c \in C$ is $G$-unstable, then since $c$ is $H$-fixed and $H$ is $G$-cr, $c$ is also $N_G(H)$-unstable. 
 Hence $C_{{\rm ss},N_G(H)} \subseteq M_{{\rm ss},G}$.
 Thus we can now apply Proposition \ref{prop:closedstrong} to deduce that $C/N_G(H)$ has closed image in $M/G$.  Since $i_G\colon \widetilde{W}/G\to M/G$ is injective (Lemma~\ref{lem:insep}), we deduce that $C/N_G(H)$ has closed image in $ \widetilde{W}/G$, as we wanted.  This allows us to draw the following commutative diagram:
 
  $$\xymatrixcolsep{5pc}\xymatrix{
C/N_G(H)\ar[d]_{}\ar[r]^{\psi_{\widetilde{W},H}}& \widetilde{W}/G\ar[d]^{(\nu_W)_G}\\
W^H/N_G(H)\ar[r]^{\psi_{W,H}}& W/G\\}$$\\

\noindent where by abuse of notation we denote the restriction of $\psi_{\widetilde{W},H}$ to $C/N_G(H)$ by the same symbol.
The leftmost vertical arrow is the isomorphism induced by the isomorphism $C\iso W^H$ above.  
The other vertical map is finite (Proposition~\ref{prop:finquotfin}) and birational (Lemma~\ref{lem:finitebirat}; recall that $W$ has good dimension).  By Theorem~\ref{thm:closed}, $\psi_{W,H}$ has closed image, and we have just argued that $\psi_{\widetilde{W},H}\left(C/N_G(H)\right)$ is closed.  
But $W= \ovl{G\cdot W^H}$, so $\psi_{W,H}$ is surjective, and it follows that $\psi_{\widetilde{W},H}\left(C/N_G(H)\right) = \widetilde{W}/G$.  
Since $\psi_{W,H}$ is quasi-finite (Theorem~\ref{thm:quasifinite}) and has singletons as generic fibres, the same is true of $\psi_{\widetilde{W},H}$.
 As $\widetilde{W}/G$ is normal, it follows from Proposition~\ref{prop:insep_ZMT} that $\psi_{\widetilde{W},H}$ is finite and bijective. 
This implies that $(\nu_W)_G\circ \psi_{\widetilde{W},H}$ is finite. Since the leftmost vertical arrow is an isomorphism, we have that $\psi_{W,H}$ is finite,
 as required.
This completes the proof of Theorem~\ref{thm:mainthm}.


\section{Separability of $\psi_{X,H}$}
\label{sec:sep}

We now consider the question of when $\psi_{X,H}$ is an isomorphism, or close to being one.  Before we state our result, we need some terminology.

\begin{defn}
 Let $H$ be a subgroup of $G$.  We say that $H$ is a {\em principal stabilizer} for the $G$-variety $X$ if there exists a nonempty open subset $U$ of $X$ such that $G_x$ is $G$-conjugate to $H$ for all $x\in U$.  We say that $H$ is a {\em principal connected stabilizer} for the $G$-variety $X$ if $H$ is connected and there exists a nonempty open subset $U$ of $X$ such that $G_x^0$ is $G$-conjugate to $H$ for all $x\in U$.  It is immediate that if $G$ permutes the irreducible components of $X$ transitively then a principal stabilizer (resp., principal connected stabilizer) is unique up to conjugacy, if one exists.
\end{defn}

In characteristic 0, principal stabilizers exist under mild hypotheses: for instance, if $X$ is smooth \cite[Prop.\ 5.3]{rich_princ} or if $X$ has good dimension \cite[Lem.\ 3.4]{LuRi}.
For a counterexample in positive characteristic, see Example~\ref{ex:noslice}.

\begin{thm}
\label{thm:normaliso}
 Let $X$ be an affine $G$-variety.  Suppose that: (a) $H$ is a principal stabilizer for $X_{\rm cl}$; (b) $H$ is $G$-cr; (c) $X/G$ and $X^H/N_G(H)$ are irreducible; and (d) $X/G$ is normal.  Then $\psi_{X,H}$ is finite and bijective.  In particular, if $\psi_{X,H}$ is separable then it is an isomorphism.
\end{thm}

Observe that this result extends a theorem of Luna and Richardson \cite[Thm.\ 4.2]{LuRi} to positive characteristic; note that in characteristic 0, a reductive group $H$ is automatically $G$-cr, $\psi_{X,H}$ is automatically separable and principal stabilizers exist, as noted above.

\begin{proof}
By Theorem~\ref{thm:mainthm}, $\psi_{X,H}$ is finite, so its fibres are finite.  To prove the first assertion of the theorem it is enough, therefore, by Proposition~\ref{prop:insep_ZMT} to show that $\psi_{X,H}$ is surjective and generic fibres of $\psi_{X,H}$ are singletons.  By hypothesis, $G\cdot X^H$ contains a nonempty open subset of $X_{\rm cl}$. The assumption that $X/G$ is irreducible implies that the action of $G$ is transitive on the irreducible components of $X_{\rm cl}$ so we can conclude that $\pi_{X,G}(X^H)= \pi_{X,G}(G\cdot X^H)$ contains a nonempty open subset of $X/G$.  As $X/G$ is irreducible, $\psi_{X,H}(X^H/N_G(H))= X/G$.  If $x$ is a stable point of $X_{\rm cl}$ and $G_x= H$ then $\psi_{X,H}^{-1}(\psi_{X,H}(\pi_{X^H,N_G(H)}(x)))$ is a singleton, by Remark~\ref{rem:singletons}.
 This proves the first assertion as the set of conjugates of such $x$ is open in $X_{\rm cl}$.  If $\psi_{X,H}$ is separable then the second assertion follows from Zariski's Main Theorem, as $X/G$ is normal.
\end{proof}

\begin{rem} 
 The assertion of Theorem~\ref{thm:normaliso} also holds by a similar argument if we replace the hypothesis that $H$ is a principal stabilizer for $X_{\rm cl}$ with the hypothesis that $H$ is a principal connected stabilizer for $X_{\rm cl}$.
\end{rem}

Next we study the separability condition.  To simplify the arguments below, we consider only the case when $X$ has good dimension for the $G$-action.

\begin{lem}
\label{lem:sepcrit}
 Suppose an affine $G$-variety $X$ has good dimension and hypotheses (a)--(c) of Theorem~\ref{thm:normaliso} hold.  Then $\psi_{X,H}$ is separable if and only if for generic $x\in X^H$, $T_x(G\cdot x)\cap T_xX^H= T_x(N_G(H)\cdot x)$.
\end{lem}

\begin{proof}
Clearly $T_x(G\cdot x)\cap T_xX^H\supseteq T_x(N_G(H)\cdot x)$, so the content here is in the reverse inclusion.  
First we claim that $X^H$ has good dimension for the $N_G(H)$-action.  To see this, observe that $\ovl{G\cdot X^H}= X$ by the surjectivity assertion of Theorem~\ref{thm:normaliso} (which does not depend on hypothesis (d)), so every closed $G$-orbit in $X$ meets $X^H$ by Lemma~\ref{lem:closedcrit}.  As $H$ is a principal stabilizer for $X$, we must have $G_x= H$ for generic $x\in X^H$, and it follows from Proposition~\ref{prop:htog} that generic $N_G(H)$-orbits in $X^H$ are closed, as required.  We now see from Remark~\ref{rem:stableopen} that 
 \begin{equation}
 \label{eqn:orbfibre}
  \pi_{X,G}^{-1}(\pi_{X,G}(x))= G\cdot x\ \mbox{and}\ \ \pi_{X^H,N_G(H)}^{-1}(\pi_{X^H,N_G(H)}(x))= N_G(H)\cdot x
 \end{equation}
 for generic $x\in X^H$.  Now $\pi_{X,G}$ and $\pi_{X^H,N_G(H)}$ are separable (Lemma~\ref{lem:fnfldquot}), and it follows from this and from Eqn.\ (\ref{eqn:orbfibre}) that for generic $x\in X^H$, $d_x\pi_{X,G}$ is surjective at $x$ with kernel $T_x(G\cdot x)$ and $d_x\pi_{X^H,N_G(H)}$ is surjective at $x$ with kernel $T_x(N_G(H)\cdot x)$.
 
 The map $\psi_{X,H}$ is surjective and finite (by Theorem~\ref{thm:mainthm}), so it is separable if and only its derivative is an isomorphism for generic points in $X^H/N_G(H)$.  The result now follows from the argument above. 
\end{proof}

Recall that a pair $(G,H)$ of reductive groups with $H\leq G$ is called a \emph{reductive pair} if $\hh = \Lie(H)$ splits off as a direct $H$-module summand of $\Gg=\Lie(G)$,
where $H$ acts via the adjoint action of $G$ on $\Gg$,
and a subgroup $A \leq G$ is called \emph{separable in $G$} if
$$
\Lie(C_G(A)) = c_{\Gg}(A) := \{X \in \Gg \mid \Ad_G(a)(X) = X \textrm{ for all } a \in A\}.
$$

\begin{prop}
 Suppose an affine $G$-variety $X$ has good dimension and hypotheses (a)--(c) of Theorem~\ref{thm:normaliso} hold.  Suppose one of the following holds:
 \begin{itemize}
\item[(i)] there exists $x\in X$ such that $G_x= H$ and there is an \'etale slice through $x$ for the $G$-action;
\item[(ii)] $H$ is separable in $G$, $(G,H)$ is a reductive pair and there exists $x\in X$ such that $G_x= H$ and $G\cdot x$ is separable.
\end{itemize}
Then $\psi_{X,H}$ is separable.
\end{prop}

\begin{proof}
 By the argument of Theorem~\ref{thm:normaliso}, $\psi_{X,H}$ is dominant.  Suppose first that (i) holds.  Let $x\in X$ with $G_x= H$ and let $S$ be an \'etale slice through $x$ for the $G$-action.  By the definition of \'etale slices and the proof of \cite[Prop.\ 8.6]{BaRi}, there exists a $G$-stable open neigbourhood $U$ of $x$ in $X$ such that $G_y\leq G_x$ for all $y\in S\cap U$ and the obvious maps $G\times (S\cap U)\ra X$ and $(S\cap U)/H\ra X/G$ are \'etale.  As $H$ is a principal stabilizer for $X$, we can assume after replacing $U$ with a smaller open set that $G_y$ is conjugate to $H$ for all $y\in S\cap U$.  We have $G_x= H$ by hypothesis, so it follows that $G_y= H$ for all $y\in S\cap U$.
 As the set of stable points of $X$ is $G$-stable, open and nonempty and the set of smooth points of $X/G$ is open and nonempty, there is a nonempty $G$-stable open subset $U_1$ of $U$ such that $G\cdot y$ is closed and $\pi_{X,G}(y)$ is a smooth point of $X/G$ for all $y\in U_1$.
 
 Since $G\cdot (S\cap U_1)$ is open and $S\cap U_1\subseteq X^H$, $G\cdot (S\cap U_1)$ contains a nonempty open subset of $X^H$.  Let $y'\in X^H\cap G\cdot (S\cap U_1)$: say, $y'= g\cdot y$ for some $y\in S\cap U_1$, $g\in G$.  Then $G_y= H$ and $G_{y'}$ is $G$-conjugate to $H$; but $y'\in X^H$, so $G_{y'}= H$.  It follows that $g\in N_G(H)$.  We deduce that $X^H\cap G\cdot (S\cap U_1)= N_G(H)\cdot (S\cap U_1)$.  So $\pi_{X^H,N_G(H)}(S\cap U_1)= \pi_{X^H,N_G(H)}(N_G(H)\cdot (S\cap U_1))$ contains a nonempty open subset of $X^H/N_G(H)$.
 
 So pick $y\in S\cap U_1$ such that $\pi_{X^H,N_G(H)}(y)$ is a smooth point of $X^H/N_G(H)$.  The map $(S\cap U)/H\ra X/G$ is \'etale, so its derivative is an isomorphism everywhere.  Hence the derivative of the map $X^H\ra X/G$ induced by $\pi_{X,G}$ is surjective at $y$.  This in turn implies that the derivative of $\psi_{X,H}$ is surjective at $\pi_{X^H,N_G(H)}(y)$.  But $\pi_{X^H,N_G(H)}(y)$ and $\pi_{X,G}(y)$ are smooth points by construction, so $\psi_{X,H}$ is separable.
 
 Now suppose that (ii) holds.  We argue along the lines of the proof of \cite[Thm.\ A]{rich1}.  Let ${\mathfrak d}$ be an $H$-module complement to $\hh$ in $\Gg$.  Let $X_0= \{x_1\in X\mid G_{x_1}= H \textrm{ and } G\cdot x_1 \textrm{ is closed and separable}\}$.  Let $x_1\in X_0$.  Then the orbit map $\kappa_{x_1}\colon G\ra G\cdot x_1$ gives an isomorphism $\phi\colon G/H\ra G\cdot x_1$.  In particular, the derivative $d_1\phi$ at $1\in G$ gives an isomorphism from $\Gg/\hh$ to the tangent space $T_{x_1}(G\cdot x_1)$, and it is easily checked that $d_1\phi$ is $H$-equivariant.  It follows that $d_1\kappa_{x_1}$ gives an isomorphism of $H$-modules from ${\mathfrak d}$ to $T_{x_1}(G\cdot x_1)$.  Now let $\beta\in T_{x_1}(G\cdot x_1)\cap T_{x_1}X^H$.  Then $\beta$ is fixed by $H$, so $\beta= d_1\kappa_{x_1}(\alpha)$ for some $\alpha\in {\mathfrak d}^H$.  As $H$ is separable in $G$, $\alpha\in {\rm Lie}(N_G(H))$.  Hence $\beta\in T_{x_1}(N_G(H)\cdot x_1)$.
 
 To finish, it is enough by Lemma~\ref{lem:sepcrit} to show that generic elements of $X^H$ belong to $X_0$.  As $H$ is a principal stabilizer for $X$ and $X$ has good dimension for the $G$-action, $G_{x_1}= H$ and $G\cdot x_1$ is closed for generic $x_1\in X^H$.  Now
 \begin{equation}
 \label{eqn:generic}
  {\rm dim}(G_{x_1})+ {\rm dim}({\rm ker}(d_1\kappa_{x_1}))\geq 2\,{\rm dim}(H)
 \end{equation}
 for all $x_1\in X^H$.  But equality holds in Eqn.\ (\ref{eqn:generic}) for $x_1= x$, so it holds for generic $x_1\in X^H$ by Lemma~\ref{lem:semicontinuity}.  This shows that $G\cdot x_1$ is separable for generic $x_1\in X^H$, so we are done.
\end{proof}

The following example shows that separability does not hold automatically under the hypotheses of Theorem~\ref{thm:normaliso}, not even when $X$ has good dimension.

\begin{ex}
 Let $G= \SL_p(k)$, where $k$ has characteristic $p$ and $p> 2$.  Let $e_1,\ldots,e_p$ be the standard basis vectors for the vector space $V:= k^p$ and let $B_0$ be the standard nondegenerate symmetric bilinear form on $k^p$ given by $B_0(e_i,e_j)= \delta_{ij}$.  Now let $Y$ be $S^2(V)^*$, the vector space of symmetric bilinear forms on $k^p$; then $G$ acts on $Y$ by $(g\cdot B)(v,w)= B(g^{-1}\cdot v, g^{-1}\cdot w)$.  If $B\in Y$ then $B$ is nondegenerate if and only if the $p\times p$ matrix with $i,j$-entry $B(e_i,e_j)$ has nonzero determinant, so the subvariety $X$ of nondegenerate forms is open and affine.  Moreover, $X$ has good dimension since the $G$-orbits on $X$ all have the same dimension.
 
 The stabilizer $G_{B_0}$ is the special orthogonal group $H:= {\rm SO}_p(k)$, and $H$ is $G$-cr as ${\rm char}(k)\neq 2$ (in fact, $H$ is contained in no proper parabolic subgroups of $G$, so $H$ is ``$G$-irreducible'').  It is easily seen that $X^H= \{cB_0\mid c\in k^*\}$ and $N_G(H)= H$; hence $N_G(H)$ acts trivially on $X^H$.  Moreover, $X= G\cdot X^H$.  Hence $H$ is a principal stabilizer and $X$ has good dimension for the action of $G$ on $X$. 
 
 Let $0\neq B= cB_0\in X^H$.  Define $\lambda\in Y(G)$ by $\lambda(a)= {\rm diag}(a^{-1},\dots, a^{-1}, a^{p-1})$ (the diagonal matrix with given entries with respect to the basis $e_1,\ldots, e_p$).  Let $B_1\in Y$ be the degenerate form given by $B_1(a_1e_1+ \cdots + a_pe_p, b_1e_1+ \cdots+ b_pe_p)= ca_pb_p$.  Then for all $a\in k^*$, $\lambda(a)\cdot B= a^2B+ (a^{2- 2p}- a^2)B_1$.  As $X$ is open in $Y$, we may identify $T_BX$ with $T_BY$.  Making the usual identification of the tangent spaces $T_1k^*$ and $T_BY$ with $k$ and $Y$, respectively, we see that
 $$ d_1\kappa_B(1)= 2B $$
(note that since ${\rm char}(k)= p$, we have $\displaystyle \left.\frac{{\rm d}}{{\rm d}a} (a^{{2- 2p}}- a^2)\right|_{a= 1}= 0$).  Now $d_1\kappa_B(1)$ belongs to $T_B(G\cdot B)$ and to $T_B(X^H)$, but not to $T_B(N_G(H)\cdot B)$ since the latter tangent space is zero.  It follows from Lemma~\ref{lem:sepcrit} that $\psi_{X,H}$ is not separable.
\end{ex}

\section{Examples}
\label{sec:ex}

The constructions in Lemma~\ref{lem:notGcrfixedpoint} demonstrate the failure of Theorem \ref{thm:mainthm} when the hypothesis of complete reducibility is removed.  In this section we provide some concrete and straightforward examples of this phenomenon.

\begin{ex}
\label{ex:sl2}
Let the characteristic be $2$ and let $\rho:\SL_2(k) \to \SL_3(k)$ be the
adjoint representation of $\SL_2(k)$.
Concretely, let
$$
e=\left(\begin{array}{cc} 0&1\\0&0 \end{array}\right),
h=\left(\begin{array}{cc} 1&0\\0&1 \end{array}\right),
f=\left(\begin{array}{cc} 0&0\\1&0 \end{array}\right)
$$
be the standard basis for $X:=\Lie(\SL_2(k))$ and let $\SL_2(k)$ act on $X$ by conjugation.
Then, with respect to this basis, we have
$$
\rho\left(\begin{array}{cc} a&b\\c&d \end{array}\right)
=
\left(\begin{array}{ccc} a^2&0&b^2\\ac&1&bd\\c^2&0&d^2 \end{array}\right).
$$
Let $H$ be the image of $\rho$ inside $G = \SL_3(k)$ with natural module $X$.
Then $H$ is reductive, but $H$ is not $G$-cr since the representation $\rho$ is not semisimple:
the $H$-fixed subspace of $X$ spanned by the vector $h$ has no $H$-stable complement.
Since $H$ is reductive, $N_G(H)^0 = H^0C_G(H)^0$.
Direct calculation shows that $C_G(H)$ is finite and hence
$N_G(H)/H$ is finite.
Now the vector $h$ is $H$-fixed but has a non-closed $G$-orbit,
since if we let $\lambda \in Y(G)$ be the cocharacter defined by
$$
\lambda(a) := \left(\begin{array}{ccc}1&0&0\\0&a &0\\0&0&a\inverse\end{array}\right)
$$
for each $a \in k^*$,
then $\lambda(a)\cdot h = ah$, so $\lim_{a\to 0} \lambda(a) \cdot h = 0$.
It is obvious that $0$ is not $G$-conjugate to $h$.
Note that the same reasoning works for any nonzero multiple of $h$.
On the other hand, the $N_G(H)$-orbit of any nonzero multiple of $h$ is finite (and hence closed),
and there are therefore infinitely many such closed $N_G(H)$-orbits.
Hence the fibre of $\psi_{X,H}$ over $\pi_{X,G}(0)$ is infinite.

Note that this example only works in characteristic $2$ because it relies on the existence of the $H$-fixed vector $h$.
This is consistent with the results above, since away from characteristic $2$ the image of the adjoint representation of $\SL_2(k)$ in $\SL_3(k)$
\emph{is} completely reducible---actually, it is irreducible---and hence is $\SL_3(k)$-cr.
\end{ex}

\begin{ex}\label{ex:harry1}
We now provide an infinite family of examples generalizing the previous one.  In these examples, $G$ is $\SL_m(k)$ acting on its natural module $X$, and $H$ is the image of some reductive group under a representation in $\SL(X)$.  
Since $G$ has only one closed orbit in $X$ (the orbit $\{0\}$), the quotient $X/G$ is just a single point.

First we consider polynomial representations of $\GL_n(k)$ where $k$ is an algebraically closed field of positive characteristic $p$. A good reference for  the polynomial representation theory of $\GL_n(k)$ is the monograph \cite{EGS}. Further details may also be found in the monograph \cite{D}. (To apply this here one should take $q=1$ in the set-up considered there.)

Let the characteristic be $p>0$ and let $G=\GL_n(k)$ be the group of $n\times n$-invertible matrices. The irreducible polynomial representations of $G$ are parametrized by partitions with at most $n$ parts. More precisely, let $\Lambda^+(n)$ be the set of partitions $\lambda=(\lambda_1,\dots,\lambda_n)$ with $\lambda_1\geq\cdots \geq \lambda_n\geq 0$. We may regard $\lambda$ as a weight of the standard maximal torus of $G$: we set $\lambda(t)= t_1^{\lambda_1}\dots t_n^{\lambda_n}$.  Then for each $\lambda\in\Lambda^+(n)$ there exists an irreducible polynomial $G$-module $L(\lambda)$  such that $L(\lambda)$ has unique highest weight $\lambda$ and 
$\lambda$ occurs as a weight with multiplicity one.
The modules $L(\lambda), \lambda\in \Lambda^+(n)$, form a complete set of pairwise non-isomorphic polynomial irreducible $G$-modules. We write $T$ for the maximal torus of $G$ consisting  of diagonal matrices and  $B$ for the subgroup of $G$ consisting of all invertible lower triangular matrices. We shall also need modules induced from $B$ to $G$. We denote by $k_\lambda$ the 1-dimensional rational $T$-module on which $t\in T$ acts as multiplication by $\lambda(t)$. The action of $T$ on $k_\lambda$ extends to an action of $B$. For each $\lambda\in\Lambda^+(n)$ the induced module $\nabla(\lambda):= \ind_B^Gk_\lambda$ is a non-zero polynomial representation of $G$.
Then $\nabla(\lambda)$ is finite-dimensional and contains the irreducible module $L(\lambda)$: in fact the $G$-socle of $\nabla(\lambda)$ is $L(\lambda)$.

We consider the induced $\GL_n(k)$-module $\nabla(n(p-1))$. We have that $\nabla(n(p-1))=S^{n(p-1)}E$, where $S^{n(p-1)}E$ is the $n(p-1)$th symmetric power of the natural $\GL_n(k)$-module $E$.

By \cite[Lem.\ 3.3]{DVD} and \cite[4.3, (10)]{D}, the $\GL_n(k)$-module $\nabla(n(p-1))$ has simple head $L(p-1,\dots,p-1)$, which is the 1-dimensional module obtained as the $(p-1)$th tensor power of the determinant module $D=L(1,\dots,1)$  of $\GL_n(k)$.
Now let $\Delta(n(p-1))$ be the Weyl module corresponding to the partition $(n(p-1))$. This is the contravariant dual of $\nabla(n(p-1))$. Since $\nabla(n(p-1))$ has simple head we get that  $\Delta(n(p-1))$ has simple socle; more precisely, 
$$
\soc_{\GL_n(k)}(\Delta(n(p-1)))=L(p-1,\dots,p-1)=D^{\otimes(p-1)}.
$$

Now consider $\Delta(n(p-1))$ as an $\SL_n(k)$-module in the usual way. 
As an $\SL_n(k)$-module, $\Delta(n(p-1))$ is the Weyl module corresponding to the dominant weight $(n(p-1))$ and by the considerations above we get that it has simple socle; in particular, $\soc_{\SL_n(k)}(\Delta(n(p-1)))=L(0)=k$ is the trivial $\SL_n(k)$-module.
Moreover, since $\Delta(n(p-1))$ is multiplicity-free as an $\SL_n(k)$-module we have that $L(0)$ appears as a composition factor of $\Delta(n(p-1))$ with multiplicity 1.

We consider the matrix representation obtained by the $\SL_n(k)$-module $\Delta(n(p-1))$. Hence we have a group homomorphism
$$\rho:\SL_n(k)\rightarrow \SL_m(k),$$
where $m=\dim (\Delta(n(p-1)))=\binom{np-1}{np-n}$.
Let $X=\Delta(n(p-1))$ and let $H$ be the image of $\rho$ inside $G= \SL_m(k)= \SL(X)$.
The previous reasoning shows that $X$ is an indecomposable $H$-module and the trivial module appears in the $H$-socle of $X$.
The group $H$ is reductive but not $G$-cr since the representation $X$ is not semisimple.

Since $H$ is reductive we have that $N_G(H)^0=H^0 C_G(H)^0$.
Moreover, $\End_H(X)=\End_{\SL_n(k)}(X)=k$ (see \cite[Prop. 2.8]{Jan}); this implies that $C_G(H)$ is finite, so $N_G(H)/H$ is finite.

Now the quotient $X^H/N_G(H)$ is infinite since $H$ fixes a full 1-dimensional subspace of $X$ and $N_G(H)/H$ is finite.
On the other hand, the quotient $X/G$ is a single point and so the morphism
$$
\psi_{X,H}\colon X^H/N_G(H)\rightarrow X/G
$$
is not a finite morphism.

Note that Example~\ref{ex:sl2} above is just this one with $p=n=2$.
\end{ex}

\begin{ex}
\label{ex:harry2}
We provide another example, this time with a symplectic group.
Let $p=2$ and consider the symplectic group $\SP_4(k)$. 
We choose the simple roots $\alpha=(2,-1)$ and $\beta=(-2,2)$.
The simple $\SP_4(k)$-module $L(0,1)$, corresponding to the dominant weight $(0,1)$, is 4-dimensional with weights $(0,1), (2,-1), (-2,1),(0,-1)$. We consider the Weyl module $\Delta(0,1)$ corresponding to $(0,1)$.  This  is  an indecomposable 5-dimensional module with simple head  $L(0,1)$ and it fits into the short exact sequence
$$
0\rightarrow k\rightarrow \Delta(0,1)\rightarrow L(0,1)\rightarrow 0,$$
where $k$ is the trivial $\SP_4(k)$-module.

Now consider the matrix representation corresponding to the $\SP_4(k)$-module $\Delta(0,1)$.
This gives a group homomorphism
$$
\rho:\SP_4(k)\rightarrow \SL_5(k).
$$
Let $X=\Delta(0,1)$ and let $H$ be the image of $\SP_4(k)$ in $G=\SL_5(k)$ with natural module $X$.
Then $X$ is an indecomposable $H$-module and the trivial module appears in the $H$-socle of $X$.
The group $H$ is reductive but not $G$-cr since the representation $X$ is not semisimple.

Since $H$ is reductive we have that $N_G(H)^0=H^0 C_G(H)^0$.
Moreover, we have that $\End_H(X)=\End_{\SP_4(k)}(X)=k$ (see \cite[Prop. 2.8]{Jan}), so the only endomorphisms of $X$ as an $H$-module are the scalars.
Since $G = \SL_5(k)$, this means that $C_G(H)$ is finite and so $N_G(H)/H$ is finite.
Now, as in our previous examples, the quotient $X^H/N_G(H)$ is infinite since $H$ fixes a full one-dimensional subspace of $X$ and $N_G(H)/H$ is finite, whereas the quotient $X/G$ is a single point.
Therefore the morphism $\psi_{X,H}\colon X^H/N_G(H)\rightarrow X/G$ is not a finite morphism.
\end{ex}

\begin{ex}
 The above examples show that if $H$ is the image of a non-completely reducible representation of a reductive group in $G= \GL(X)$ or $\SL(X)$ then the conclusion of Theorem~\ref{thm:mainthm} can fail.  On the other hand, if $H$ is the image of a completely reducible representation then we get an easy representation-theoretic proof of Theorem~\ref{thm:mainthm} in this special case, as follows.  
If the representation is trivial (of any dimension), so that $H$ is the trivial group, then $X^H = X$ and $N_G(H) = G$, so the map $\psi_{X,H}$ is the identity map.
If the representation is non-trivial and irreducible, then
$X^H = \{0\}$ and the map $\psi_{X,H}\colon X^H/N_G(H)\ra X/G$ is just the map from a singleton set to a singleton set and hence is finite.  
If the representation is non-trivial and completely reducible but not irreducible then $X^H$ has an $H$-complement in $X$: say, $X = X^H \oplus W$.  The centre of the Levi subgroup of $G$ corresponding to the given decomposition normalizes $H$ and acts as scalars on $X^H$, so $X^H/N_G(H)$ is again a singleton set and $\psi_{X,H}$ is finite.
\end{ex}

\section{Double cosets}
\label{sec:dblecosets}

In this section we consider a separate but related problem, using techniques from earlier sections.  Fix a reductive group $G$, and reductive subgroups $H$ and $K$ of $G$.  The group $H\times K$ acts on $G$ by the formula $(h,k)\cdot g= hgk^{-1}$; the orbits of the action are the $(H,K)$-double cosets and we call this action the {\em double coset action}.  The stabilizer $(H\times K)_g$ is given by $\{(h,g^{-1}hg)\mid h\in H\cap gKg^{-1}\}$.  We are interested in the following question: when does $G$ have good dimension for the double coset action?  
Note that, again, in characteristic $0$ this problem was solved by Luna in \cite{luna00}; he showed using \'etale slices that $G$ always has good dimension for the double coset action. 
The problem of translating Luna's results to positive characteristic was also studied by Brundan \cite{brundan1, brundan2, brundan3, brundan5}, who considered in particular the question of when there is a dense double coset in $G$.
Our main result gives a necessary and sufficient condition for $G$ to have good dimension for the double coset action in terms of the stabilizers of the action.

\begin{thm}
\label{thm:closedcrit}
 Let $G$ be connected.  The following are equivalent:
 \begin{itemize}
 \item[(i)] $G$ has good dimension for the $(H\times K)$-action;
 \item[(ii)] generic stabilizers of $H\times K$ on $G$ are reductive;
 \item[(iii)] $H\cap gKg^{-1}$ is reductive for generic $g\in G$.
 \end{itemize}
\end{thm}

\begin{rems}
\label{rem:closedcrit}
 (i). It follows from \cite[Thm.\ 1.1]{martin3} that in order to show that generic stabilizers are reductive, it is enough to show that $(H\times K)_g$ has minimal dimension and is reductive for some $g\in G$.

 (ii). Work of Popov \cite{popov} implies that if a connected semisimple group $G$ acts on a smooth irreducible affine variety $V$ and the divisor class group ${\rm Cl}(V)$ has no elements of infinite order then generic orbits of $G$ on $V$ are closed if and only if generic stabilizers of $G$ on $V$ are reductive.  By work of Tange \cite[Thm.\ 1.1]{tange}, if $G$ is connected then ${\rm Cl}(G)$ has no elements of infinite order, so Theorem~\ref{thm:closedcrit} follows if $H$ and $K$ are connected and semisimple.
\end{rems}

 We need some preparatory results and notation.  
 First, given a cocharacter $\tau = (\lambda,\mu)\in Y(H\times K)$ and $g\in G$, we say that 
 $\tau$ \emph{destabilizes $g$} if $\lim_{a\to 0} \tau(a)\cdot g = \lim_{a\to0} \lambda(a)g\mu(a)^{-1}$ exists.
 Given $g\in G$, define a homomorphism $\widehat{\phi}_g\colon G\ra G\times G$ by $\widehat{\phi}_g(g')= (g',g^{-1}g'g)$.  A short calculation shows that $\widehat{\phi}_g$ induces an isomorphism $\phi_g\colon H\cap gKg^{-1}\ra (H\times K)_g$.  This shows that (ii) and (iii) of Theorem~\ref{thm:closedcrit} are equivalent.  Moreover, given $g\in G$ we define an isomorphism of varieties $r_g\colon G\ra G$ by $r_g(g')= g'g^{-1}$ and an isomorphism of algebraic groups $\psi_g\colon H\times K\ra H\times gKg^{-1}$ by $\psi_g(h,k)= (h,gkg^{-1})$; then $r_g$ is a $\psi_g$-equivariant map from the $(H\times K)$-variety $G$ to the $(H\times gKg^{-1})$-variety $G$, where we let $H\times gKg^{-1}$ act on $G$ by the double coset action.

\begin{lem}
\label{lem:basechange}
 Let $g\in G$.  Then $G$ has good dimension for the $(H\times K)$-action if and only if $G$ has good dimension for the $H\times gKg^{-1}$-action, and generic stabilizers of $H\times K$ on $G$ are reductive if and only if generic stabilizers of $H\times gKg^{-1}$ on $G$ are reductive.
\end{lem}

\begin{proof}
 The $\psi_g$-equivariance of $r_g$ implies that $(H\times gKg^{-1})\cdot r_g(g')= r_g((H\times K)\cdot g')$ and $r_g((H\times K)_{g'})= (H\times gKg^{-1})_{r_g(g')}$ for all $g'\in G$.  The result follows.
\end{proof}

In the special case when $A$ is reductive, the next result is \cite[Lem.\ 4.1]{martin1}.  We take the opportunity to correct the proof given in {\em loc.\ cit.}

\begin{lem}
\label{lem:extnclosed}
 Let $k'$ be an algebraically closed extension field of $k$.  Let $A$ be a linear algebraic group acting on an affine variety $X$, and let $A'$ (resp.\ $X'$) be the group (resp.\ variety) over $k'$ obtained from $K$ (resp.\ $X$) by extension of scalars.  Let $x\in X$.  Then:
\begin{itemize} 
 \item[(i)] ${\rm dim}_{k'}(A'_x)= {\rm dim}_k(A_x)$ and ${\dim}_{k'}(A'\cdot x)= {\rm dim}_k(A\cdot x)$;
 
\item[(ii)] $A'\cdot x$ is closed in $X'$ if and only if $A\cdot x$ is closed in $X$.
\end{itemize}
\end{lem}

\begin{proof}
 We regard $X$ as a subset of $X'$ and $A$ as a subgroup of $A'$ in the obvious way.  The orbit map $\kappa_x\colon A'\ra A'\cdot x$ is defined over $k$, so the closure $\ovl{A'\cdot x}$ (in $X'$) is $k$-defined \cite[Cors.\ AG.14.5 and AG.14.6]{borel}.  This implies that $\ovl{A'\cdot x}\cap X= \ovl{A\cdot x}$, where the RHS is the closure in $X$.  The stabilizer $A'_x$ is $k$-defined---in fact, $A'_x$ is naturally isomorphic to the group over $k'$ obtained from $A_x$ by extension of scalars.  Hence ${\rm dim}_{k'}(A'_x)= {\rm dim}_k(A_x)$.  This proves the first assertion of (i), and the second follows immediately.
 
 Let $r= {\dim}_{k'}(A'\cdot x)= {\rm dim}_k(A\cdot x)$.  Set $X'_t= \{y'\in X'\mid {\rm dim}_{k'}(A'_{y'})\geq t\}$ and $X_t= \{y\in X\mid {\rm dim}_k(A_y)\geq t\}$ for $t\geq 0$.  Then $X'_t$ and $X_t$ are closed in $X'$ and $X$, respectively, and it follows from the proof of \cite[Lem.~3.7(c)]{newstead} that $X'_t$ is $k$-defined.  By (i), $X_t= X'_t\cap X$.  Now $\ovl{A'\cdot x}$ is the union of $A'\cdot x$ with certain other $A'$-orbits, each of which has dimension strictly less than $r$, and likewise for $\ovl{A\cdot x}$.  Hence $A'\cdot x$ (resp., $A\cdot x$) is closed if and only if $\ovl{A'\cdot x}\cap X'_{r+1}= \emptyset$ (resp., $\ovl{A\cdot x}\cap X_{r+1}= \emptyset$).  But $\ovl{A'\cdot x}\cap X'_{r+1}$ is $k$-defined and $k$ is algebraically closed, so $\ovl{A'\cdot x}\cap X'_{r+1}$ is empty if and only if $(\ovl{A'\cdot x}\cap X'_{r+1})\cap X= \ovl{A\cdot x}\cap X_{r+1}$ is empty.  Part (ii) now follows.
\end{proof}

\begin{lem}
\label{lem:extn}
 Assume $G$ is connected.  Let $k'$ be an algebraically closed extension field of $k$ and let $G'$, $H'$ and $K'$ be the algebraic groups over $k'$ obtained from $G$, $H$ and $K$, respectively, by extension of scalars.  Then:
 \begin{itemize}
 \item[(i)] generic stabilizers of $H'\times K'$ on $G'$ are reductive if and only if generic stabilizers of $H\times K$ are reductive;
 \item[(ii)] $G'$ has good dimension for the $(H'\times K')$-action if and only if $G$ has good dimension for the $(H\times K)$-action.
 \end{itemize}
\end{lem}

\begin{proof}
 We can regard $G$, $H$ and $K$ as dense subgroups of $G'$, $H'$ and $K'$, respectively.  If $g\in G$ then $(H'\times K')_g$ is isomorphic to the group obtained from $(H\times K)_g$ by extension of scalars, so $(H'\times K')_g$ is reductive if and only if $(H\times K)_g$ is reductive.  By Lemma~\ref{lem:extnclosed}, $(H'\times K')\cdot g$ is closed in $G'$ if and only if $(H\times K)\cdot g$ is closed in $G$, and ${\rm dim}((H'\times K')_g)$ is minimal if and only if ${\rm dim}((H\times K)_g)$ is minimal, so $g$ is a stable point for the $(H'\times K')$-action if and only if it is a stable point for the $(H\times K)$-action.  The union of the stable $(H\times K)$-orbits is open in $G$, and likewise for $H'\times K'$ and $G'$ (Lemma~\ref{rem:stableopen}).  The union of the $(H\times K)$-orbits of minimum dimension having reductive stabilizer is open in $G$, and likewise for $H'\times K'$ and $G'$ \cite[Thm.\ 1.1]{martin3}.  Putting these facts together, we obtain the desired result.
\end{proof}

\begin{lem}
\label{lem:destab}
 Let $\lambda\in Y(H)$, $\mu\in Y(K)$.  Given $g\in G$ such that 
 $(\lambda,\mu)$ destabilizes $g$, set  
 $g_0:= \lim_{a\to 0} \lambda(a)g\mu(a)^{-1}$ and let $u=g g_0^{-1}$.  Then $\mu= g_0^{-1}\cdot \lambda$ and $u\in R_u(P_\lambda)$.
\end{lem}

\begin{proof}
 Since $g_0$ is obtained as a limit along $(\lambda,\mu)$, we have that $(\lambda,\mu)$ fixes $g_0$, so $\lambda(a)g_0\mu(a)^{-1} = g_0$ for all $a\in k^*$. Rearranging, we see that $\mu= g_0^{-1}\cdot \lambda$.  
 Now for all $a \in k^*$,
 $$
 \lambda(a)g\mu(a)^{-1}
 = \lambda(a)ug_0\mu(a)^{-1}
 = \lambda(a)ug_0(g_0^{-1} \lambda(a)^{-1}g_0)
 = \lambda(a)u\lambda(a)^{-1}g_0.
 $$  
 As $\lim_{a\to 0} \lambda(a)g\mu(a)^{-1}= g_0$, 
 it follows that $\lim_{a\to 0} \lambda(a)u\lambda(a)^{-1}= 1$, 
 so $u\in R_u(P_\lambda)$.
\end{proof}

\begin{lem}
\label{lem:semisimple}
 Assume $G$ is connected.  Let $G_1= G/Z(G)^0$, let $\sigma\colon G\ra G_1$ be the canonical projection and set $H_1= \sigma(H)$ and $K_1= \sigma(K)$.  Then for all $g\in G$:\begin{itemize}
 \item[(i)] $(H\times K)_g$ is reductive if and only if $(H_1\times K_1)_{\sigma(g)}$ is reductive;
 \item[(ii)] if $(H_1\times K_1)\cdot \sigma(g)$ is closed then $(H\times K)\cdot g$ is closed.
\end{itemize}
\end{lem}

\begin{proof}
(i). Let $\widetilde{H}= HZ(G)^0$ and let $\widetilde{K}= KZ(G)^0$.  Let $A= (\sigma\times \sigma)^{-1}((H_1\times K_1)_{\sigma(g)})$, a subgroup of $\widetilde{H}\times \widetilde{K}$.  Define $\psi\colon A\ra G$ by $\psi(\widetilde{h}, \widetilde{k})= g^{-1}\widetilde{h}g\widetilde{k}^{-1}$.  A short calculation shows that $\psi$ gives a homomorphism from $A$ to $Z(G)^0$, with kernel $(H\times K)_g$.  Moreover, $\sigma\times \sigma$ gives an epimorphism from $A$ to $(H_1\times K_1)_{\sigma(g)}$, with kernel $Z(G)^0\times Z(G)^0$.  Part (i) now follows.

(ii). Let $g\in G$ and suppose $(H_1\times K_1)\cdot\sigma(g)$ is closed.  
Let $(\lambda,\mu)\in Y(H\times K)$ such that 
$g':= \lim_{a\to 0} \lambda(a)g\mu(a)^{-1}$ exists.  
Set $g_1= \sigma(g)$, $g_1'= \sigma(g')$.  
Let $\lambda_1= \sigma\circ\lambda\in Y(H_1)$ and $\mu_1 = \sigma\circ\mu \in Y(K_1)$; then $g_1'= \lim_{a\to 0} \lambda_1(a)g_1\mu_1(a)^{-1}$.  
By hypothesis, $g_1'$ is $(H_1\times K_1)$-conjugate to $g_1$.  
Now the group $Z(G)^0$ acts on $G$ by right inverse multiplication, and we can identify $\sigma$ with the canonical projection to the quotient.  
The orbits of $Z(G)^0$ all have the same dimension, so $\sigma$ is a geometric quotient.  Moreover, the $Z(G)^0$-action commutes with the $(H\times K)$-action, so $H\times K$ acts on $G_1$.  
By construction, $(h,k)\cdot \sigma(x)= (\sigma(h),\sigma(k))\cdot \sigma(x)$ for all $x\in G$, $h\in H$ and $k\in K$.  In particular, $g_1'$ is $(H\times K)$-conjugate to $g_1$.  
It follows from \cite[Cor.\ 3.5(ii)]{GIT} that $g'$ is $(H\times K)$-conjugate to $g$.  Hence $(H\times K)\cdot g$ is closed.  
This proves (ii).
\end{proof}

\begin{lem}
\label{lem:oneclass}
 Suppose $G$, $H$ and $K$ are connected.  Let $\lambda\in Y(H)$. Suppose there exists a nonempty subset $C$ of $G$ such that 
 $(H\times K)\cdot C$ is open and has the following property: 
 for all $g\in C$, there exists $\tau_g= (\lambda,\mu_g)\in Y(H\times K)$ 
 such that $\tau_g$ destabilizes $g$.  
 \begin{itemize}
 \item[(i)] There exists $g_0\in G$ such that 
 $\lambda\in Y(g_0Kg_0^{-1})$ and $(H\times g_0Kg_0^{-1})\cdot P_\lambda$ 
 is dense in $G$.  
 Moreover, for all $g\in P_\lambda$, the cocharacter $(\lambda,\lambda)$ of 
 $H\times g_0Kg_0^{-1}$ destabilizes $g$.
 \item[(ii)] Suppose in addition that $\tau_g$ fixes $g$ for all $g\in C$.  Then $(H\times g_0Kg_0^{-1})\cdot L_\lambda$ is dense in $G$, and $(\lambda,\lambda)$ fixes every $l\in L_\lambda$.
\end{itemize}
\end{lem}

\begin{proof}
 Fix $v\in C$ and let 
 $v_0:= \lim_{a\to 0} \tau_v(a)\cdot v= \lim_{a\to 0} \lambda(a)v\mu_v(a)$.  
 Then $\lambda= v_0\cdot \mu_v$ by Lemma~\ref{lem:destab}, so 
 $\lambda\in Y(v_0Kv_0^{-1})$.  
 The equivariance of $r_{v_0}$ implies that for any $w\in C$, 
 $(\lambda, v_0\cdot \mu_w)\in Y(H\times v_0Kv_0^{-1})$ destabilizes
 $wv_0^{-1}$ to $w_0v_0^{-1}$, where $w_0:= \lim_{a\to 0} \tau_w(a)\cdot w$.  
 By Lemma~\ref{lem:basechange}, we can replace $K$ with
 $v_0Kv_0^{-1}$ and $C$ with $Cv_0^{-1}$.  
 So without loss we assume that $\lambda\in Y(K)$.
 
 Let $g\in C$.  By hypothesis, 
 $\tau_g = (\lambda,\mu_g)$ destabilizes $g$.
 As ${\rm im}(\lambda)$ is contained in $K$, there exists $k\in K$ such that 
 $\mu:= k\cdot \mu_g$ commutes with $\lambda$.  
 Set $g_1= gk^{-1} = (1,k)\cdot g$, so that $(\lambda,\mu)$ destabilizes $g_1$.  
 Finally, set $g_2= \lim_{a\to 0} \lambda(a)g_1\mu(a)^{-1}$.  
 Then $\lambda= g_2\cdot \mu$ by Lemma~\ref{lem:destab}.  
 Fix a maximal torus $T$ of $G$ such that $\lambda,\mu\in Y(T)$ and let 
 $n_1,\ldots, n_r\in N_G(T)$ be a set of representatives for the Weyl group $N_G(T)/T$.
 Now $g_2Tg_2^{-1}$ is a maximal torus of $L_\lambda$, so by conjugacy of maximal tori
 in $L_\lambda$, we have $xg_2Tg_2^{-1}x^{-1}= T$ for some $x\in L_\lambda$.  
 Then $xg_2= tn_i$ for some $i$ and some $t \in T$, so $g_2= ln_i$, where $l:= x^{-1}t\in L_\lambda$.  
 By Lemma~\ref{lem:destab}, we have $g_1= ug_2= uln_i$ for some 
 $u\in R_u(P_\lambda)$, so $g_1\in P_\lambda n_i$ and $g = g_1k \in (H\times K)\cdot (P_\lambda n_i)$. 
 Since $g\in C$ was arbitrary, it now follows that $\bigcup_{i=1}^r (H\times K)\cdot (P_\lambda n_i)$ 
 contains $(H\times K)\cdot C$ and, since $G$ is connected, 
 $(H\times K)\cdot (P_\lambda n_i)$ is dense in $G$
 for at least one $i$. 
 Note also that $\lambda= g_2\cdot \mu= ln_i\cdot \mu$, so 
 $\mu= n_i^{-1}l^{-1}\cdot \lambda= n_i^{-1}\cdot \lambda$, 
 so $\lambda= n_i\cdot \mu\in Y(n_iKn_i^{-1})$.

 Keeping the notation in the previous paragraph, for each $i$, 
 $(H\times K)\cdot (P_\lambda n_i)$ is constructible, 
 so $(H\times K)\cdot (P_\lambda n_i)$ is either dense or contained 
 in a proper closed subset of $G$.  
 Thus the union of those subsets $(H\times K)\cdot (P_\lambda n_i)$ 
 that are dense contains an open subset of $G$; note also that this union is $(H\times K)$-stable.
 Since $(H\times K)\cdot C$ is open,
 we can find $g'\in C$ such that for any $i$, 
 if $g'\in (H\times K)\cdot (P_\lambda n_i)$ then $(H\times K)\cdot (P_\lambda n_i)$ is dense.  
 By the arguments in the paragraph above applied to $g'$, 
 there exists $i$ such that $g'\in (H\times K)\cdot(P_\lambda n_i)$ and for this $i$ we have 
 $\lambda= n_i\cdot \mu\in Y(n_iKn_i^{-1})$; moreover,
 $(H\times K)\cdot (P_\lambda n_i)$ is dense by construction.  It follows that $(H\times n_iKn_i^{-1})\cdot P_\lambda= r_{n_i}((H\times K)\cdot(P_\lambda n_i))$ is dense in $G$, 
so the first assertion of part (i) follows with $g_0 = n_i$.
It is obvious that $(\lambda,\lambda)$ destabilizes $g$ for all 
$g \in P_\lambda$, so we have proved part (i).
 
 If $g\in C$ and $\tau_g$ fixes $g$ then $(\lambda,\mu)$ fixes $g_1$, so $g_1= g_2\in L_\lambda n_i$ for some $i$.  The first assertion of (ii) follows by a similar argument to that above but applied to $\bigcup_{i=1}^r (H\times K)\cdot (L_\lambda n_i)$, and the second assertion is again obvious.
\end{proof}

\begin{proof}[Proof of Theorem~\ref{thm:closedcrit}]
 We have shown already that (ii) and (iii) are equivalent, so it is enough to prove that (i) and (ii) are equivalent.  First note that for any $g\in G$, $(H\cap K)\cdot g$ is closed if and only if $(H\cap K)^0\cdot g= (H^0\cap K^0)^0\cdot g$ is closed, and $H\cap gKg^{-1}$ is reductive if and only if $(H\cap gKg^{-1})^0= (H^0\cap gK^0g^{-1})^0$ is reductive, which is the case if and only if $H^0\cap gK^0g^{-1}$ is reductive.  Hence we can assume that $H$ and $K$ are connected.  Moreover, we can assume by Lemma~\ref{lem:extn} that $k$ is uncountable.
 
  The implication (i)$\implies$(ii) follows immediately from Lemma~\ref{lem:quotientorbit}(ii).  For the reverse implication, we use induction on ${\rm dim}(G)$.  Suppose generic stabilizers are reductive.  The result is immediate if ${\rm dim}(G)= 0$.  If $G$ is not semisimple then let $G_1$, $\sigma$, $H_1$ and $K_1$ be as in Lemma~\ref{lem:semisimple}.  Then generic stabilizers of $H_1\times K_1$ on $G_1$ are reductive, by Lemma~\ref{lem:semisimple}(i).  Since ${\rm dim}(G_1)< {\rm dim}(G)$, it follows by induction that generic orbits of $H_1\times K_1$ on $G_1$ are closed.  Part (ii) of Lemma~\ref{lem:semisimple} now implies that generic orbits of $H\times K$ on $G$ are closed, so we are done.  Hence we can assume that $G$ is semisimple.
  
 First we consider the case when generic stabilizers of $H\times K$ on $G$ are positive-dimensional.  Then all stabilizers of $H\times K$ on $G$ are positive-dimensional, by semi-continuity of stabilizer dimension.  For each $g\in G$ such that $(H\times K)_g$ is reductive, choose a nontrivial cocharacter $\tau_g\in Y((H\times K)_g)$.  The fixed point set $G^{\tau_g}:= G^{{\rm im}(\tau_g)}$ is closed, so $C_g:= (H\times K)\cdot G^{\tau_g}$ is constructible.  Since generic stabilizers of $H\times K$ on $G$ are reductive, the constructible sets $C_g$ for $g\in G$ such that $(H\times K)_g$ is reductive cover an open dense subset $U$ of $G$, by \cite[Thm.\ 1.1]{martin3}.  There are only countably many of these sets, as $H\times K$ has only countably many conjugacy classes of cocharacters.  By \cite[Cor.\ 2.5]{martin3}, $C_{\widetilde{g}}$ is dense in $G$ for some $\widetilde{g}\in G$.  Hence there exists $\tau= (\lambda,\mu)\in Y(H\times K)$ such that for generic $g\in G$, $g$ is fixed by an $(H\times K)$-conjugate of $\tau$.  It follows from Lemma~\ref{lem:oneclass} that for some $g_0\in G$, $\lambda\in Y(g_0Kg_0^{-1})$ and $(H\times g_0Kg_0^{-1})\cdot L_\lambda$ is dense in $G$.  By Lemma~\ref{lem:basechange}, there is no harm in assuming that $g_0Kg_0^{-1}= K$---i.e., that $\lambda \in Y(K)$ and $(H\times K)\cdot L_\lambda$ is dense in $G$---and we shall do this for notational convenience.
 
To prove that generic $(H\times K)$-orbits on $G$ are closed, it is therefore enough to show that $(H\times K)\cdot l$ is closed for generic $l\in L_\lambda$.  Let $H_2= L_\lambda(H)$ and let $K_2= L_\lambda(K)$; then $H_2\times K_2= L_{(\lambda,\lambda)}(H\times K)$.  Consider the double coset action of $H_2\times K_2$ on $L_\lambda$.  Let $l\in L_\lambda$.  Then $(\lambda,\lambda)$ fixes $l$, so $(H_2\times K_2)_l= L_{(\lambda,\lambda)}((H\times K)_l)$, which is reductive if $(H\times K)_l$ is.  Hence generic stabilizers of $(H_2\times K_2)$ on $L_\lambda$ are reductive.  As $G$ is semisimple, ${\rm dim}(L_\lambda)< {\rm dim}(G)$, so generic $(H_2\times K_2)$-orbits on $L_\lambda$ are closed by induction.  It follows from Remark~\ref{rems:failure}(iii) that $(H\times K)\cdot l$ is closed for generic $l\in L_\lambda$, so we are done as $(H\times K)\cdot L_\lambda$ is dense in $G$.
 
 Now consider the case when generic stabilizers of $H\times K$ on $G$ are finite.  Suppose generic $(H\times K)$-orbits on $G$ are not closed.  Then $G_{\rm cl}$ is a proper closed subset of $G$, so the union of the non-closed orbits contains a nonempty open subset of $G$.  For each $g\in G$ such that $(H\times K)\cdot g$ is not closed, choose nontrivial $\tau_g\in Y(H\times K)$ such that $\tau_g$ destabilizes $g$.  By an argument similar to the one in the positive-dimensional case above, there exist $\lambda\in Y(H)$ and $g_0\in G$ such that $\lambda\in Y(g_0Kg_0^{-1})$ and $(H\times g_0Kg_0^{-1})\cdot P_\lambda$ is dense in $G$.  As before, we can assume that $g_0Kg_0^{-1}= K$.  Now $R_u(P_{-\lambda})(H)P_\lambda(H)$ and $P_\lambda(K) R_u(P_{-\lambda})(K)$ are nonempty open subsets of $H$ and $K$ respectively \cite[Prop.~14.21(iii)]{borel}, so $R_u(P_{-\lambda})(H)P_\lambda R_u(P_{-\lambda})(K)$ is dense in $G$, as $HP_\lambda K$ is.  It follows that ${\rm dim}(R_u(P_{-\lambda})(H))+ {\rm dim}(R_u(P_{-\lambda})(K))+ {\rm dim}(P_\lambda)\geq {\rm dim}(G)$, so ${\rm dim}(R_u(P_{-\lambda})(H))+ {\rm dim}(R_u(P_{-\lambda})(K))\geq {\rm dim}(G)- {\rm dim}(P_\lambda)= {\rm dim}(R_u(P_\lambda))$.
 
 By hypothesis, we can choose $g\in P_\lambda$ such that $(H\times K)_g$ is finite.  Write $g= ul$, where $l= L_\lambda$ and $u\in R_u(P_\lambda)$; then $l= \lim_{a\to 0} (\lambda, \lambda)(a)\cdot g$.  We show that $l$ is $(H\times K)$-conjugate to $g$.  Consider the double coset action of $R_u(P_\lambda(H))\times R_u(P_\lambda(K))$ on $G$.  Let $O= (R_u(P_\lambda(H))\times R_u(P_\lambda(K)))\cdot g$ and consider the orbit map $\kappa_g\colon R_u(P_\lambda(H))\times R_u(P_\lambda(K))\ra O$ given by $\kappa_g(h,k)= hgk^{-1}$.  It is clear that $O\subseteq R_u(P_\lambda)l$.  Note that $O$ is closed, since $O$ is the orbit of an action of a unipotent group on an affine variety \cite[Prop.\ 4.10]{borel}.  The stabilizer of $g$ in $R_u(P_\lambda(H))\times R_u(P_\lambda(K))$ is finite, since $(H\times K)_g$ is finite, so $O$ has dimension ${\rm dim}(R_u(P_\lambda(H)))+ {\rm dim}(R_u(P_\lambda(K)))$.  Now ${\rm dim}(R_u(P_\lambda(H)))+ {\rm dim}(R_u(P_\lambda(K)))= {\rm dim}(R_u(P_{-\lambda})(H))+ {\rm dim}(R_u(P_{-\lambda})(K))\geq {\rm dim}(R_u(P_\lambda))$, and since $O$ is closed, this forces $O$ to be the whole of $R_u(P_\lambda)l$.  Hence there exists $(h,k)\in R_u(P_\lambda(H))\times R_u(P_\lambda(K))$ such that $(h,k)\cdot g= l$, as required.
 
 Now $(H\times K)_l$ is finite, since $l$ is $(H\times K)$-conjugate to $g$.  But $(\lambda,\lambda)$ fixes $l$, a contradiction.  We deduce that generic $(H\times K)$-orbits on $G$ are closed after all.  This completes the proof.
\end{proof}

\begin{rem}
 One can prove the following more general statement of Theorem~\ref{thm:closedcrit} for non-connected reductive $G$.  Let $G_1,\ldots, G_r$ be the minimal subsets of $G$ having the property that each $G_i$ is $(H\times K)$-stable and contains some connected component of $G$.  Each $G_i$ is a union of certain connected components of $G$; if $H$ and $K$ are connected then the $G_i$ are precisely the connected components of $G$.  Here is our result: for each $i$, $G_i$ has good dimension for the $(H\times K)$-action if and only if generic stabilizers of $H\times K$ on $G_i$ are reductive if and only if $H\cap gKg^{-1}$ is reductive for generic $g\in G_i$.  To see this, note first that we can assume that $H$ and $K$ are connected, by the proof of Theorem~\ref{thm:closedcrit}; hence we can assume that each $G_i$ is a connected component of $G$.  We can now choose $g\in G$ such that $G_ig= G^0$, and use the map $r_g$ to translate the case of $G_i$ into the case of the connected group $G^0$ (cf.\ the proof of Lemma~\ref{lem:basechange}).  We leave the details to the reader.
\end{rem}

We record a useful corollary.

\begin{cor}
\label{cor:torus}
 Suppose one of $H$ and $K$ is a torus.  Then $G$ has good dimension for the $(H\times K)$-action.
\end{cor}

\begin{proof}
 This is immediate from Theorem~\ref{thm:closedcrit}, since any subgroup of a torus is reductive.
\end{proof}

We now consider a concrete example; our methods allow us to deal with arbitrary characteristic.  Note that we use Theorem~\ref{thm:mainthm} in parts (a) and (b) below.

\begin{ex}
\label{ex:B2}
 Let $G$ be simple of type $B_2$ and fix a maximal torus $T$ of $G$.  Let $A$ be the subgroup of $G$ generated by the long root groups with respect to $T$.  If $p= 2$ then let $B$ be the subgroup of $G$ generated by the short root groups with respect to $T$.  The groups $A$ and $B$ are normalized by $N_G(T)$.
 
 (a). Let $p$ be arbitrary and let $H= K= A$.  Since ${\rm dim}(G)= 10$ and ${\rm dim}(H)= {\rm dim}(K)= 6$, ${\rm dim}(H\times K)_g\geq 2$ for all $g\in G$, with equality if and only if $(H\times K)\cdot g$ is dense in $G$.  Let $\lambda\in Y(T)$ be nontrivial.  We show first that for generic $l\in L_\lambda$, $(H\times K)\cdot l$ is closed.  If $L_\lambda= T$ or $L_\lambda$ is a long-root Levi subgroup (that is, a Levi subgroup $L$ such that $[L,L]$ is the subgroup of type $A_1$ corresponding to some long root) then $L_\lambda\leq A$, so $(H\times K)\cdot l= A$ is closed.  Note that in this case, $(H\times K)\cdot L_\lambda= A$ is not dense in $G$.
 
 So suppose $L_\lambda$ is a short-root Levi subgroup.  As in the positive-dimensional case in the proof of Theorem~\ref{thm:closedcrit}, it is enough to show that $(L_\lambda(H)\times L_\lambda(K))\cdot l$ is closed for generic $l\in L_\lambda$.  But this follows from Corollary~\ref{cor:torus}, since $L_\lambda(H)\times L_\lambda(K)= T\times T$ is a torus.  Moreover, in this case the quotient space $L_\lambda/(T\times T)$ is positive-dimensional, as ${\rm dim}(T\times T)= 4= {\rm dim}(L_\lambda)$ and $(T\times T)_l$ has dimension at least 1 for all $l\in L$ (since $(\lambda,\lambda)$ fixes $l$).  It follows that the quotient space $G/(H\times K)$ is positive-dimensional.  To see this, let $S$ be the image of $(\lambda, \lambda)$; note that $T\times T\leq N_{H\times K}(S)\leq N_{H\times K}(T\times T)$.  Now consider the maps $L_\lambda/(T\times T)\ra L_\lambda/N_{H\times K}(S)\ra G/(H\times K)$.  The first map is finite as $T\times T$ has finite index in $N_{H\times K}(S)$, while the second is finite by Theorem~\ref{thm:mainthm} (applied to the subgroup $S$ of $H\times K$), so ${\rm dim}(G/(H\times K))\geq {\rm dim}(L_\lambda/(T\times T))\geq 1$, as claimed.
 
 Next we show that for generic $g\in G$, $(H\times K)_g$ contains a nontrivial torus.  Suppose not.  Then for generic $g\in G$, $(H\times K)_g^0$ is a unipotent subgroup of $H$ of dimension at least 2, so $(H\times K)_g^0$ is a maximal unipotent subgroup of $H$ and has dimension 2.  But then the orbit $(H\times K)\cdot g$ is dense in $G$, so $G/(H\times K)$ is a single point, which is a contradiction.
 
 It follows from the proof of the positive-dimensional case of Theorem~\ref{thm:closedcrit} that $(H\times n_iKn_1^{-1})\cdot L_\lambda$ is dense in $G$ for some nontrivial $\lambda\in Y(T)$ and some $i$.  But $N_G(T)$ normalizes $K$, so $(H\times K)\cdot L_\lambda$ is dense in $G$ (and hence $L_\lambda$ is a short-root Levi subgroup of $G$).  We deduce from the discussion above that generic $(H\times K)$-orbits in $G$ are closed.  Moreover, we see that the map $L_\lambda/(T\times T)\ra G/(H\times K)$ is finite and dominant, hence surjective.  A simple calculation shows that generic stabilizers of $T\times T$ on $L_\lambda$ have dimension 1, so ${\rm dim}(L_\lambda/(T\times T))= 1$, which implies that ${\rm dim}(G/(H\times K))= 1$.  Hence generic stabilizers of $H\times K$ on $G$ are reductive groups of dimension 3.  It follows that for generic $g\in G$, $(H\times K)_g^0$ is of type $A_1$. 

\smallskip
 (b). Let $p= 2$ and let $H= K= B$.  Then generic orbits of $H\times K$ on $G$ are closed and for generic $g\in G$, $(H\times M)_g^0$ is of type $A_1$.  The proof is similar to case (a).
 
\smallskip 
 (c). Let $p= 2$, let $H= A$ and let $K= B$.  Consider the stabilizer $(H\times K)_1$, which is isomorphic via the map $\phi_1$ to $H\cap K$.  It is easily seen that $H\cap K= T$, so $(H\times K)_1= \{(t,t)\mid t\in T\}$.  It follows that ${\rm dim}((H\times K)\cdot 1)= {\rm dim}(H\times K)- {\rm dim}(T)= 12- 2= 10= {\rm dim}(G)$, so $(H\times K)\cdot 1$ is dense in $G$ and generic stabilizers have dimension 2 and are reductive.  It follows from Theorem~\ref{thm:closedcrit} that generic orbits are closed. Hence $(H\times K)\cdot 1$ is closed and $H\times K$ acts transitively on $G$.  (This conclusion also follows from \cite[Thm.\ A]{brundan3}, since $A$ and $B$ are maximal connected subgroups of $G$.)
\end{ex}

We finish the section with a further example.

\begin{ex}
\label{ex:E7}
 Suppose $p\neq 2$ and let $G$ be simple and of rank $r$.  Let $\tau\in {\rm Aut}(G$) be an involution that inverts a maximal torus of $G$---such a $\tau$ always exists, by \cite[Lem.\ 3.6]{BGS}---and let $H= K= C_G(\tau)$.  Then $(H\times K)_g= H\cap gKg^{-1}$ is a finite group of order $2^r$ for generic $g\in G$ \cite[Thm.\ 9]{BGS}, so $G$ has good dimension for the $H\times K$-action, by Theorem~\ref{thm:closedcrit}.  (Note that in \cite[Sec.\ 3.2]{BGS} one considers the action of $H$ by left multiplication on $G/K$ rather than the double coset action of $H\times K$ on $G$, but the arguments carry over easily to our setting.  See also \cite[Ex.\ 8.4]{martin3}.)
 
 We now consider a striking feature of this example: namely that, although generic stabilizers for the double coset action are nontrivial, there is a unique $(H\times K)$-orbit $O$ consisting of elements with trivial stabilizer \cite[Thm.\ 9]{BGS} (compare Example~\ref{ex:noslice}).  If $p= 0$ then $O$ cannot be closed: for otherwise there would exist an \'etale slice through any element of $O$, so every element in some open neighbourhood of $O$ would have trivial stabilizer by the argument of Example~\ref{ex:noslice}, a contradiction.  More generally, for arbitrary $p\neq 2$ the argument of \cite[Ex.\ 8.4]{martin3} implies that $G$ has a principal stabilizer $A$ which is a finite group of order $2^r$, and it follows from Theorem~\ref{thm:mainthm} that if $g'\in G$ and $(H\times K)\cdot g'$ is closed then $(H\times K)_{g'}$ contains a conjugate of $A$.  Hence we see again that $O$ cannot be closed.
 
 We give a direct proof of this.  The orbit $O$ is of the form $(H\times K)\cdot g$, where $g\in G$ has the property that $u:= \tau \tau^g$ is a regular unipotent element of $G$ and $u$ is inverted by $\tau$ (see \cite[Prop.\ 3.1]{BGS}).  In fact, we can choose $g$ to be a regular unipotent element of $G$ such that $g^2= u$ and $\tau$ inverts $g$ (take $g$ to be $u^s$ if $p> 0$, where $2s\equiv 1\ {\rm mod}\ |u|$).  Set $U= \ovl{\langle g\rangle}$; then $\tau$ normalizes $U$, as $\tau$ inverts $g$.  There exists $\lambda\in Y(G)$ such that $\lim_{a\to 0} \lambda(a)g'\lambda(a)^{-1}= 1$ for all $g'\in U$.  We can choose $\lambda$ to be optimal in the sense of \cite[Defn.\ 4.4 and Thm.\ 4.5]{GIT} (cf.\ Section~\ref{subsec:reductiveaffine}).  Then $\tau$ normalises $P_\lambda$.  Now $N_{{\rm Aut}(G)}(P_\lambda)$ is an R-parabolic subgroup of the reductive group ${\rm Aut}(G)$ \cite[Prop.\ 5.4(a)]{martin1}, so $N_G(P_\lambda)= P_\mu$ for some $\mu\in Y(G)$.  As $\tau\in P_\mu$ and $\langle \tau\rangle$ is linearly reductive, we can choose $\mu$ to centralize $\tau$: that is, we can choose $\mu$ to belong to $Y(H)$.
 
 Let $\sigma= (\mu, \mu)\in Y(H\times K)$.  Then $\lim_{a\to 0} \sigma(a)\cdot g= \lim_{a\to 0} \mu(a)g\mu(a)^{-1}= 1$ since $U\leq R_u(P_\lambda)= R_u(P_\mu)$.  But clearly $1\not\in O$, so $O$ is not closed, as claimed.
\end{ex}

\section{Applications to $G$-complete reducibility}
\label{sec:Gcr}

We finish with some applications of ideas from Sections~\ref{sec:prelims} and \ref{sec:dblecosets} to $G$-complete reducibility.  Our next lemma gives a basic structural result about $G$ and its subgroups which can quickly be
proven using the framework we have now set up; the setting is as in Section~\ref{sec:dblecosets} but more general, since we allow one of the subgroups to be non-reductive (cf.\ \cite{brundan4}).  The argument used is taken from the proof of \cite[Kap. III.2.5, Satz 2]{kraft};
note that although the reference \cite{kraft} works with groups and varieties defined
over the complex numbers, many of the arguments are completely general.
For convenience, we reproduce the details here.

\begin{lem}\label{lem:HK}
Suppose $K$ is a subgroup of $G$ and let $H$ be a reductive subgroup of $G$ that contains a maximal torus of $K$.
Then $HK$ is a closed subset of $G$.
\end{lem}

\begin{proof}
First suppose that $K$ is unipotent.
The quotient $X=G/H$ is affine and $H$ is the stabilizer in $G$ of the point $x = \pi_{G,H}(1) \in X$.
Since $K$ is unipotent, and all orbits for unipotent groups on affine varieties are closed \cite[Prop.\ 4.10]{borel},
$K\cdot x$ is closed, so $KH$ (and hence $HK$) is closed in $G$ by Lemma~\ref{lem:closedhorbits}.

Now, in the general case, let $T$ be a maximal torus of $K$ contained in $H$ and let $B$ be a Borel subgroup of $K$ containing $T$ with unipotent radical $U$.
Then $UH = BH$ is closed in $G$ by the first paragraph, and the following argument from \cite[Kap. III.2.5, Satz 2]{kraft} gives us what we want.
We have a sequence of morphisms
$$
K\times G \overset{\phi}{\longrightarrow} K\times G \xrightarrow{\rho=\pi_{K,B} \times {\rm id}} K/B \times G \overset{{\rm pr}_2}{\longrightarrow} G
$$
where $\phi(g',g):=(g',g'g)$ for $g' \in K$, $g \in G$, $\pi_{K,B}$ is the quotient morphism $K \to K/B$ and ${\rm pr}_2$ is the projection of $K/B \times G$ onto the second factor.
Let $Y = K \times BH$.
Since $BH$ is closed in $G$, $Y$ is closed in $K \times G$.
Since $\phi$ is an isomorphism of varieties, $\phi(Y)$ is closed in $K\times G$ and therefore $\rho(\phi(Y))$ is closed in $K/B \times G$.
Finally, since $K/B$ is complete, ${\rm pr}_2(\rho(\phi(Y)))$ is closed in $G$.
But it is easy to see that ${\rm pr}_2(\rho(\phi(Y)))=KH$, so we are done.
\end{proof}

The lemma above allows a quick proof of the following result.

\begin{prop}\label{prop:Hxmaxrank}
Suppose $G$ is reductive, $X$ is an affine $G$-variety and $x \in X$. 
If $H$ is a reductive subgroup of $G$ containing a maximal torus of $G_x$, then $H\cdot x$ is closed in $G\cdot x$.
In particular, if $G\cdot x$ is closed in $X$ then $H\cdot x$ is closed in $X$.
\end{prop}

\begin{proof}
By Lemma \ref{lem:HK}, under the given hypotheses, $HG_x$ is closed in $G$.
Hence, by Lemma \ref{lem:closedhorbits}, $H\cdot x$ is closed in $G\cdot x$.
\end{proof}

\begin{rem}\label{rems:maxrank}
The argument of Lemma \ref{lem:HK} is used in \cite[Kap.\ III, 2.5, Folgerung 3]{kraft} to show that if $X$ is affine and $G_x$ contains
a maximal torus of $G$, then $G\cdot x$ is closed, a result which has obvious similarities to Proposition \ref{prop:Hxmaxrank}.
\end{rem}

\begin{cor}
Suppose $H$ and $K$ are reductive subgroups of the reductive group $G$.
If $H\cap K$ contains a maximal torus of $H$ or $K$, then $HK$ is closed in $G$ and $H\cap K$ is a reductive group.
\end{cor}

\begin{proof}
Without loss, suppose $H\cap K$ contains a maximal torus of $K$.
The first conclusion is a special case of Lemma \ref{lem:HK}.
For the second, apply Proposition \ref{prop:Hxmaxrank} to the action of $G$ and $H$ on the quotient $G/K$.
Since $G/K$ is affine and $H\cdot \pi_{G,K}(1)$ is closed in $G/K$,
this orbit is also affine and hence the stabilizer $H_{\pi_{G,K}(1)}= H\cap K$ is reductive by Lemma \ref{lem:quotientorbit}(ii).
\end{proof}

A theme running through \cite{BMR} and subsequent papers on complete reducibility by the same authors is the following general question:
if $A$ and $H$ are subgroups of $G$ with $A\subseteq H$ and $H$ reductive,
what conditions ensure that if $A$ is $G$-cr then $A$ is $H$-cr, and vice versa?
Because of the link between complete reducibility and closed orbits in $G^n$ explained in Section~\ref{subsec:cr} above,
this is readily seen to be a special case of the general questions considered in this paper.
Since this was one of the original motivations for the work presented here, we briefly record some of the translations of our main results into the
language of complete reducibility and give a couple of other consequences in this setting.

First note that Proposition \ref{prop:Hxmaxrank} specializes to \cite[Prop.\ 3.19]{BMR} in the setting of complete reducibility:
that is, with notation as just set up, if $H$ also contains a maximal torus of $C_G(A)$ and $A$ is $G$-cr, then $A$ is $H$-cr.
More generally, we have:

\begin{prop}
\label{prop:AHG}
Suppose $H$ is a reductive subgroup of $G$, and let $A$ be a subgroup of $H$.
\begin{itemize}
\item[(i)] If $A$ is $H$-cr, then
$HC_G(A)$ is closed in $G$.
\item[(ii)] If
$A$ is $G$-cr, then $A$ is $H$-cr if and only if $HC_G(A)$ is closed in $G$.
\end{itemize}
\end{prop}

\begin{proof}
(i). Let $\mathbf{a} \in H^n$ be a generic tuple for $A$.
Suppose $A$ is $H$-cr; then $H\cdot\mathbf{a}$ is closed in $H^n$.
Since
$H^n$ is closed in $G^n$, $H\cdot\mathbf{a}$ is closed in $G\cdot\mathbf{a}$.
Therefore, by Lemma \ref{lem:closedhorbits}(i), $HG_\mathbf{a} = HC_G(A)$ is closed in $G$.

(ii). Using a generic tuple for $A$ again, this becomes a direct application of Lemma \ref{lem:closedhorbits}(ii).
\end{proof}

The notions of reductive pairs from \cite[$\S$3]{rich2} and separability from \cite[Def.\ 3.27]{BMR} have proved useful in the study of complete reducibility:
see \cite[$\S$3.5]{BMR}, \cite{BMRT}, \cite{BHMR} for example.
Recall that the definitions of reductive pair and separable subgroup were given in Section~\ref{sec:sep}.  We have the following result:

\begin{prop}
\label{prop:redpair}
Suppose $(G,H)$ is a reductive pair.
Let $A$ be a separable subgroup of $G$ contained in $H$. Then $HC_G(A)$ is closed in $G$.
\end{prop}

\begin{proof}
Let $\mathbf{a} \in H^n$ be a generic tuple for $A$.
Then $C_G(A) = G_\mathbf{a}$, and since $A$ is separable in $G$, the orbit $G\cdot\mathbf{a}$ is separable.
Now Richardson's ``tangent space argument'' \cite[$\S$3]{rich2} (generalized to $n$-tuples in \cite{slodowy}) shows that
$G\cdot\mathbf{a} \cap H^n$ decomposes into finitely many $H$-orbits, each of which is closed in $G\cdot\mathbf{a} \cap H^n$.
Since one of these orbits is $H\cdot \mathbf{a}$, we can conclude that $H\cdot\mathbf{a}$ is closed in $G\cdot\mathbf{a} \cap H^n$,
and hence in $G\cdot\mathbf{a}$.
Therefore, $HG_\mathbf{a} = HC_G(A)$ is closed in $G$ by Lemma \ref{lem:closedhorbits}(i).
\end{proof}

\begin{rem}
Note that \emph{every} pair $(G,H)$ of reductive groups with $H \leq G$ is a reductive pair in characteristic $0$
and the separability hypothesis is also automatic.
In characteristic $p>0$, every subgroup of $G$ is separable as long as $p$ is ``very good'' for $G$; see \cite[Thm.\ 1.2]{BMRT}.
\end{rem}

As a final remark, we note that there are Lie algebra analogues of Propositions~\ref{prop:AHG} and \ref{prop:redpair}, where we replace the subgroup $A$ with a Lie subalgebra of $\Lie(H)$.
For details of how to make such translations, see \cite[$\S$5]{GIT}, for example.

\end{document}